\documentclass[a4paper, 10pt]{amsart}

\usepackage{amsmath,amssymb, hyperref, color}

\newtheorem{theorem}{Theorem}[section]
\newtheorem{lemma}[theorem]{Lemma}

\newtheorem{proposition}[theorem]{Proposition}
\newtheorem{conjecture}[theorem]{Conjecture}

\newtheorem{problem}[theorem]{Problem}

\theoremstyle{definition}

\newtheorem{definition}[theorem]{Definition}

\newcommand{\N}{\mathbb N}
\newcommand{\Z}{\mathbb Z}
\newcommand{\R}{\mathbb R}
\newcommand{\Q}{\mathbb Q}
\newcommand{\F}{\mathbb F}

\newcommand{\red}{{\text{\rm red}}}

\newcommand{\BF}{\text{\rm BF}}

 \DeclareMathOperator{\ord}{ord}
 
 \DeclareMathOperator{\Int}{Int}
 \DeclareMathOperator{\supp}{supp}

\newcommand{\DP}{\negthinspace : \negthinspace}

\renewcommand{\P}{\mathbb P}
\renewcommand{\t}{\, | \,}
\renewcommand{\time}{\negthinspace \times \negthinspace}

\numberwithin{equation}{section}

\begin{document}

\title[Systems of sets of lengths:  transfer Krull monoids versus weakly Krull monoids]{Systems of sets of lengths: \\ transfer Krull monoids versus weakly Krull monoids}

\address{University of Graz,  \\
Institute for Mathematics  \\
Heinrichstra{\ss}e 36\\
8010 Graz, Austria}

\email{alfred.geroldinger@uni-graz.at, qinghai.zhong@uni-graz.at}

\address{Universit{\'e} Paris 13 \\ Sorbonne Paris Cit{\'e} \\ LAGA, CNRS, UMR 7539,  Universit{\'e} Paris 8\\ F-93430, Villetaneuse, France \\ and \\ Laboratoire Analyse, G{\'e}om{\'e}trie et Applications (LAGA, UMR 7539) \\ COMUE  Universit{\'e} Paris Lumi{\`e}res \\  Universit{\'e} Paris 8, CNRS \\  93526 Saint-Denis cedex, France} \email{schmid@math.univ-paris13.fr}

\author{Alfred Geroldinger  and Wolfgang A. Schmid and Qinghai Zhong}

\thanks{This work was supported by
the Austrian Science Fund FWF, Project Number P28864-N35, and the ANR project Caesar, project number ANR-12-BS01-0011}



\begin{abstract}
Transfer Krull monoids are monoids which allow a weak transfer homomorphism to a commutative Krull monoid, and hence the system of sets of lengths of a transfer Krull monoid coincides with that of the associated commutative Krull monoid. We unveil a couple of new features of the system of sets of lengths of transfer Krull monoids over finite abelian groups $G$, and we provide a complete description of the system for all groups $G$ having Davenport constant $\mathsf D (G)=5$ (these are the smallest groups for which no such descriptions were known so far). Under reasonable algebraic finiteness assumptions, sets of lengths of transfer Krull monoids and of weakly Krull monoids satisfy the Structure Theorem for Sets of Lengths. In spite of this common feature we demonstrate that systems of sets of lengths for a variety of classes of weakly Krull monoids are different from the system of sets of lengths of any transfer Krull monoid.
\end{abstract}

\maketitle

\medskip
\section{Introduction} \label{1}
\medskip

By an atomic monoid we mean a cancellative semigroup with unit element such that every nonunit can be written as a finite product of irreducible elements. Let $H$ be an atomic monoid. If $a \in H$ is a nonunit and $a=u_1 \cdot \ldots \cdot u_k$ is a factorization of $a$ into $k$ irreducible elements, then $k$ is called a factorization length and the set $\mathsf L (a) \subset \N$ of all possible factorization lengths  is called the set of lengths of $a$. Then $\mathcal L (H) = \{ \mathsf L (a) \mid a \in H \}$ is the system of sets of lengths of $H$. Under a variety of noetherian conditions on $H$ (e.g., $H$ is the monoid of nonzero elements of a commutative noetherian domain) all sets of lengths are finite. Furthermore, if there is some element $a \in H$ with $|\mathsf L (a)| > 1$, then $|\mathsf L(a^N)|>N$ for all $N \in \N$. Sets of lengths (together with invariants controlling their structure, such as elasticities and sets of distances) are a well-studied means of describing the arithmetic structure of monoids (\cite{Ge-HK06a, C-F-G-O16}).

Let $H$ be a transfer Krull monoid. Then, by definition, there is a weak transfer homomorphism $\theta \colon H \to \mathcal B (G_0)$, where $\mathcal B (G_0)$ denotes the monoid of zero-sum sequences over a subset $G_0$ of an abelian group, and hence $\mathcal L (H) = \mathcal L \big( \mathcal B (G_0) \big)$. A special emphasis has always been on the case where $G_0$ is a finite abelian group. Thus let $G$ be a finite abelian group and we use the abbreviation $\mathcal L (G) = \mathcal L \big( \mathcal B (G) \big)$. It is well-known that sets of lengths in $\mathcal L (G)$ are highly structured (Proposition \ref{3.2}), and the standing conjecture is that the system $\mathcal L (G)$ is characteristic for the group $G$. More precisely, if $G'$ is a finite abelian group such that $\mathcal L (G)= \mathcal L (G')$, then $G$ and $G'$ are isomorphic (apart from two well-known trivial pairings; see  Conjecture \ref{3.4}). This conjecture holds true, among others,  for groups $G$ having rank at most two, and its proof uses deep results from additive combinatorics which are not available for general groups. Thus there is a need for studying $\mathcal L (G)$ with a new approach. In Section \ref{3}, we unveil a couple of properties of the system $\mathcal L (G)$ which are first steps on a new way towards Conjecture \ref{3.4}.

In spite of all abstract work on systems $\mathcal L (G)$, they have been written down explicitly only for groups $G$ having Davenport constant $\mathsf D (G)\le 4$, and this is not difficult to do (recall that  a group $G$ has Davenport constant $\mathsf D(G)\le 4$ if and only if either $|G|\le 4$ or $G$ is an elementary $2$-group of rank three). In Section \ref{4} we determine the systems  $\mathcal L (G)$ for all groups $G$ having Davenport constant $\mathsf D (G)=5$.

Commutative Krull monoids are the classic examples of transfer Krull monoids. In recent years a wide range of monoids and domains has been found which are transfer Krull but which are not commutative Krull monoids. Thus the question arose which monoids $H$ have systems $\mathcal L (H)$ which are  different from systems of sets of lengths of transfer Krull monoids. Commutative $v$-noetherian weakly Krull monoids and domains are the best investigated class of monoids beyond commutative Krull monoids (numerical monoids as well as one-dimensional noetherian domains are $v$-noetherian weakly Krull). Clearly, weakly Krull monoids can be half-factorial and half-factorial monoids are transfer Krull monoids. Similarly, it can happen both for weakly Krull monoids as well as for transfer Krull monoids that all sets of lengths are arithmetical progressions with difference $1$. Apart from such extremal cases,  we show  in Section \ref{5} that systems of sets of lengths of a variety of classes of weakly Krull monoids are different from the system of sets of lengths of any transfer Krull monoid.

\medskip
\section{Background on sets of lengths} \label{2}
\medskip

We denote by $\N$ the set of positive integers, and for real numbers $a, b \in \R$, we denote by $[a,b] = \{ x \in \Z \mid a \le x \le b\}$ the discrete interval between $a$ and $b$, and by an interval we always mean a finite discrete interval of integers.

Let $A, B \subset \Z$ be subsets of the integers. Then $A+B = \{a+b \mid a \in A, b \in B\}$ is the sumset of $A$ and $B$. We set $-A= \{-a \mid a \in A\}$ and for an integer $m \in \Z$, $m+A = \{m\}+A$ is the shift of $A$ by $m$. For $m \in \N$, we denote by $mA = A+ \ldots + A$ the $m$-fold subset of $A$ and by $m \cdot A = \{ma \mid a \in A\}$ the dilation of $A$ by $m$. If $A \subset \N$, we denote by $\rho (A) = \sup A/\min A \in \Q_{\ge 1} \cup \{\infty\}$ the elasticity of $A$ and we set $\rho ( \{0\})=1$.
A positive integer $d \in \N$ is called a distance of $A$ if there are $a, b \in A$ with $b-a=d$ and the interval $[a,b]$ contains no further elements of $A$. We denote by $\Delta (A)$ the set of distances of $A$. Clearly, $\Delta (A)=\emptyset$ if and only if $|A|\le 1$, and $A$ is an arithmetical progression if and only if $|\Delta (A)|\le 1$.

Let $G$ be an additive abelian group. An (ordered) family $(e_i)_{i \in I}$ of elements of $G$ is said to be \
{\it independent} \ if $e_i \ne 0$ for all $i \in I$ and, for every
family $(m_i)_{i \in I} \in \Z^{(I)}$,
\[
\sum_{i \in I} m_ie_i =0 \qquad \text{implies} \qquad m_i e_i =0 \quad \text{for all} \quad i \in I\,.
\]
A family $(e_i)_{i \in I}$ is called a {\it basis} for $G$ if   $e_i \ne 0$ for all $i \in I$ and
 $G = \bigoplus_{i \in I} \langle e_i \rangle$. A subset $G_0 \subset G$ is said to be independent if the tuple $(g)_{g \in G_0}$ is independent. For every prime $p \in \P$, we denote by $\mathsf r_p (G)$ the $p$-rank of $G$.

\medskip
\noindent
{\bf Sets of Lengths.}
 We say that a semigroup $S$ is cancellative if for all elements $a,b,c \in S$, the equation $ab=ac$ implies $b=c$ and the equation $ba=ca$ implies $b=c$. Throughout this manuscript, a monoid means a cancellative semigroup with unit element, and we will use multiplicative notation.

 Let $H$ be a monoid.
 An element $a \in H$ is said to be invertible if there exists an element $a'\in H$ such that $aa'=a'a=1$. The set of invertible elements of $H$ will be denoted by $H^{\times}$, and we say that $H$ is reduced if $H^{\times}=\{1\}$. For a set $P$, we denote by $\mathcal F (P)$ the free abelian monoid with basis $P$. Then every $a \in \mathcal F (P)$ has a unique representation in the form
\[
a = \prod_{p \in P} p^{\mathsf v_p (a)} \,,
\]
where $\mathsf v_p  \colon \mathcal F (P) \to \N_0$ denotes the $p$-adic exponent.

An element $a \in H$ is called {\it irreducible} (or an {\it atom}) if $a \notin H^{\times}$ and if, for all $u, v \in H$, $a = u v$ implies that $u \in H^{\times}$ or $v \in H^{\times}$. We denote by $\mathcal A (H)$ the set of atoms of $H$. The monoid $H$ is said to be {\it atomic} if every $a \in H \setminus H^{\times}$ is a product of finitely many atoms of $H$.
If $a \in H$ and $a = u_1 \cdot \ldots \cdot
u_k$, where $k \in \mathbb N$ and $u_1, \ldots, u_k \in \mathcal A
(H)$, then we say that $k$ is the {\it length} of the factorization.
For $a \in H \setminus H^{\times}$, we call
\[
\mathsf L_H (a) = \mathsf L (a) = \{ k \in \mathbb N \mid a \
\text{has a factorization of length} \ k \} \subset \N
\]
the {\it set of lengths} of $a$. For convenience, we set $\mathsf L
(a) = \{0\}$ for all $a \in H^{\times}$. By definition, $H$ is
atomic if and only if $\mathsf L (a) \ne \emptyset$ for all $a \in
H$. Furthermore, $\mathsf L (a) = \{1\}$ if and only if $a \in \mathcal A (H)$ if and only if $1 \in \mathsf L (a)$. If $a, b \in H$, then  $\mathsf L (a) + \mathsf L (b) \subset \mathsf L (a b)$. We call
\[
\mathcal L (H) = \{ \mathsf L (a) \mid a \in H \}
\]
the {\it system of sets of lengths} of $H$.  We say that $H$ is {\it half-factorial} if $|L|=1$ for every $L \in \mathcal L (H)$. If $H$ is atomic, then $H$ is either half-factorial or for every $N \in \N$ there is an element $a_N \in H$ such that $|\mathsf L (a_N)| >N$ (\cite[Lemma 2.1]{Ge16c}). We say that $H$ is a BF-monoid if it is atomic and all sets of lengths are finite. Let
\[
\Delta (H) = \bigcup_{L \in \mathcal L (H)} \Delta (L) \ \subset \ \N
\]
denote the {\it set of  distances} of $H$, and if $\Delta (H)\ne \emptyset$, then $\min \Delta (H) = \gcd \Delta (H)$. We denote by $\Delta_1 (H)$ the set of all $d \in \N$ with the following property:
\begin{itemize}
\item[] For every $k \in \N$ there exists an $L \in \mathcal L (H)$ of the form $L = L' \cup \{y, y+d, \ldots, y+kd\} \cup L''$ where $y \in \N$ and $L', L'' \subset \Z$ with $\max L' < y$ and $y+kd < \min L''$.
\end{itemize}
By definition, $\Delta_1 (H)$ is a subset of $\Delta (H)$.
For every $k \in \N$ we define the $k$th {\it elasticity} of $H$. If $H=H^{\times}$, then we set $\rho_k (H) = k$, and if $H \ne H^{\times}$, then
\[
\rho_k (H) = \sup \{\sup L \mid k \in L \in \mathcal L (H) \} \in \N \cup \{\infty\} \,.
\]
The invariant
\[
\rho (H) = \sup \{ \rho (L) \mid L \in \mathcal L (H) \} = \lim_{k \to \infty}\frac{\rho_k(H)}{k} \in \R_{\ge 1} \cup \{\infty\}
\]
is called the {\it elasticity} of $H$ (see \cite[Proposition 2.4]{Ge16c}).
Sets of lengths of all monoids, which are in the focus of the present paper, are highly structured (see Proposition \ref{3.2} and Theorems \ref{5.5} -- \ref{5.8}). To summarize the relevant concepts, let $d \in \N$, \ $M \in \N_0$ \ and \ $\{0,d\} \subset \mathcal D
\subset [0,d]$. A subset $L \subset \Z$ is called an {\it almost
arithmetical multiprogression} \ ({\rm AAMP} \ for
      short) \ with \ {\it difference} \ $d$, \ {\it period} \ $\mathcal D$,
      \  and \ {\it bound} \ $M$, \ if
\[
L = y + (L' \cup L^* \cup L'') \, \subset \, y + \mathcal D + d \Z
\]
where \ $y \in \mathbb Z$ is a shift parameter,
\begin{itemize}
\item  $L^*$ is finite nonempty with $\min L^* = 0$ and $L^* =
       (\mathcal D + d \Z) \cap [0, \max L^*]$, and

\item  $L' \subset [-M, -1]$ \ and \ $L'' \subset \max L^* + [1,M]$.
\end{itemize}
We say that \ {\it the Structure Theorem for Sets of Lengths} \ holds
for a monoid $H$  if $H$ is atomic and there exist some $M \in
\N_0$ and a finite nonempty set $\Delta \subset \N$ such that
every $L \in \mathcal L(H)$ is an {\rm AAMP} with some difference $d
\in \Delta$ and bound $M$.

\medskip
\noindent
{\bf Monoids of zero-sum sequences.} We discuss a monoid having a combinatorial flavor whose universal role in the study of sets of lengths will become evident at the beginning of the next section. Let $G$ be an additive abelian group and $G_0 \subset G$ a subset. Then $\langle G_0 \rangle$ denotes the subgroup generated by $G_0$, and we set $G_0^{\bullet}= G_0 \setminus \{0\}$. In additive combinatorics, a {\it sequence} (over $G_0$) means a finite sequence of terms from $G_0$ where repetition is allowed and the order of the elements is disregarded, and (as usual) we consider sequences as elements of the free abelian monoid with basis $G_0$. Let
\[
S = g_1 \cdot \ldots \cdot g_{\ell} = \prod_{g \in G_0} g^{\mathsf v_g (S)} \in \mathcal F (G_0)
\]
be a sequence over $G_0$. We set $-S = (-g_1) \cdot \ldots \cdot (-g_{\ell})$, and  we call
\begin{itemize}
\item $ \supp (S)  = \{g \in G \mid \mathsf v_g (S) > 0 \} \subset G \ \text{the \ {\it support} \ of \ $S$}\,$,
\item $ |S|  = \ell = \sum_{g \in G} \mathsf v_g (S) \in \mathbb N_0 \ \text{the \ {\it length} \ of \ $S$} \,$,
\item $ \sigma (S)  = \sum_{i = 1}^l g_i \ \text{the \ {\it sum} \ of \
$S$} \,$,
\item $\Sigma (S)  = \Big\{ \sum_{i \in I} g_i \mid \emptyset \ne I \subset [1,\ell] \Big\} \ \text{ the \ {\it set of
subsequence sums} \ of \ $S$} \,$,
\item $\mathsf k (S)  = \sum_{i = 1}^l \frac{1}{\ord(g_i)} \ \text{the \ {\it cross number} \ of \
$S$} \,$.
\end{itemize}
The sequence $S$ is said to be
\begin{itemize}
\item {\it zero-sum free} \ if \ $0 \notin \Sigma (S)$,

\item a {\it zero-sum sequence} \ if \ $\sigma (S) = 0$,

\item a {\it minimal zero-sum sequence} \ if it is a nontrivial zero-sum
      sequence and every proper  subsequence is zero-sum free.
\end{itemize}
The set of zero-sum sequences $\mathcal B (G_0) = \{S \in \mathcal F (G_0) \mid \sigma (S)=0\} \subset \mathcal F (G_0)$ is a submonoid, and the set of minimal zero-sum sequences is the set of atoms of $\mathcal B (G_0)$.
For any arithmetical invariant $*(H)$ defined for a monoid $H$, we write $*(G_0)$ instead of $*(\mathcal B (G_0))$. In particular, $\mathcal A (G_0) = \mathcal A (\mathcal B (G_0))$ is the set of atoms of $\mathcal B (G_0)$, $\mathcal L (G_0)=\mathcal L (B(G_0))$ is the system of sets of lengths of $\mathcal B (G_0)$, and so on. Furthermore, we say that $G_0$ is half-factorial if the monoid $\mathcal B (G_0)$ is half-factorial. We denote by
\[
\mathsf D (G_0) = \sup \{ |S| \mid S \in \mathcal A (G_0) \} \in \N_0 \cup \{\infty\}
\]
the {\it Davenport constant} of $G_0$. If $G_0$ is finite, then $\mathsf D (G_0)$ is finite. Suppose that $G$ is finite, say $G \cong C_{n_1} \oplus \ldots \oplus C_{n_r}$, with $r \in \N_0$, $1 < n_1 \t \ldots \t n_r$, then $r = \mathsf r (G)$ is the rank of $G$, and we have
\begin{equation} \label{inequality-Davenport}
1 + \sum_{i=1}^r (n_i-1) \le \mathsf D (G) \le |G| \,.
\end{equation}
If $G$ is a $p$-group or $\mathsf r (G) \le 2$, then $1 + \sum_{i=1}^r (n_i-1) = \mathsf D (G)$. Suppose that $|G| \ge 3$. We will use that $\Delta (G)$ is an interval  with $\min \Delta (G)=1$ (\cite{Ge-Yu12b}), and that, for all $k \in \N$,
\begin{equation} \label{inequality-elasticity}
\begin{aligned}
\rho_{2k} (G) & =  k \mathsf D (G),  \\
 k \mathsf D (G) +1 \le \rho_{2k+1} (G)  & \le k \mathsf D (G) + \lfloor \mathsf D (G)/2 \rfloor,  \quad \text{and} \\
 \rho (G) & = \mathsf D (G)/2 \,,
\end{aligned}
\end{equation}
(\cite[Section 6.3]{Ge-HK06a}).

\medskip
\section{Sets of lengths of transfer Krull monoids} \label{3}
\medskip

Weak transfer homomorphisms play a critical role in factorization theory, in particular in all studies of sets of lengths.  We refer to \cite{Ge-HK06a} for a detailed presentation of transfer homomorphisms in the commutative setting. Weak transfer homomorphisms (as defined below) were introduced in \cite[Definition 2.1]{Ba-Sm15} and transfer Krull monoids were introduced in \cite{Ge16c}.

\medskip
\begin{definition} \label{3.1}
Let $H$ be a monoid.
\begin{enumerate}
\item A monoid homomorphism $\theta \colon H \to B$ to an atomic monoid $B$ is called a {\it weak transfer homomorphism} if it has the following two properties:
\begin{description}
\item[{\bf \phantom{W}(T1)}] $B = B^{\times} \theta (H) B^{\times}$ and $\theta^{-1} (B^{\times})=H^{\times}$.
\item[{\bf (WT2)}] If $a \in H$, $n \in \N$, $v_1, \ldots, v_n \in \mathcal A (B)$ and $\theta (a) = v_1 \cdot \ldots \cdot v_n$, then there exist $u_1, \ldots, u_n \in \mathcal A (H)$ and a permutation $\tau \in \mathfrak S_n$ such that $a = u_1 \cdot \ldots \cdot u_n$ and $\theta (u_i) \in B^{\times} v_{\tau (i)} B^{\times}$ for each $i \in [1,n]$.
\end{description}

\smallskip
\item  $H$ is said to be a  {\it transfer Krull monoid} (over $G_0$) it there exists a weak transfer homomorphism $\theta \colon H \to \mathcal B (G_0)$ for a subset $G_0$ of an abelian group $G$. If $G_0$ is finite, then we say that $H$ is a {\it transfer Krull monoid of finite type}.
\end{enumerate}
\end{definition}

If $R$ is a domain and $R^{\bullet}$ its monoid of cancellative elements, then we say that $R$ is a transfer Krull domain (of finite type) if $R^{\bullet}$ is a transfer Krull monoid (of finite type).
Let $\theta \colon H \to B$ be a weak transfer homomorphism between atomic monoids. It is easy to show that for all $a \in H$ we have $\mathsf L_H (a) = \mathsf L_B ( \theta (a))$ and hence $\mathcal L (H) = \mathcal L (B)$. Since monoids of zero-sum sequences are BF-monoids, the same is true for transfer Krull monoids.

Let $H^*$ be a commutative Krull monoid (i.e., $H^*$ is commutative, completely integrally closed, and $v$-noetherian). Then there is a weak transfer homomorphism $\boldsymbol \beta \colon H^* \to \mathcal B (G_0)$ where $G_0$ is a subset of the class group of $H^*$. Since
monoids of zero-sum sequences are  commutative Krull monoids and since the composition of weak transfer homomorphisms is a weak transfer homomorphism again, a monoid  is a transfer Krull monoid if and only if it allows a weak transfer homomorphism to a commutative Krull monoid.   In particular,  commutative Krull monoids are  transfer Krull monoids. However, a transfer Krull monoid need neither be commutative nor $v$-noetherian nor completely integrally closed. To give a noncommutative example, consider a bounded HNP (hereditary noetherian prime) ring $R$. If every stably free left $R$-ideal is free, then its multiplicative monoid of cancellative elements  is a transfer Krull monoid (\cite{Sm17a}). A class of  commutative weakly Krull domains which are transfer Krull but not  Krull  will be given in Theorem \ref{5.8}. Extended lists of commutative Krull monoids and of transfer Krull monoids, which are not commutative Krull, are given in \cite{Ge16c}.

The next proposition summarizes some key results on the structure of sets of lengths of transfer Krull monoids.

\medskip
\begin{proposition} \label{3.2}~
\begin{enumerate}
\item Every transfer Krull monoid of finite type satisfies the Structure Theorem for Sets of Lengths.

\smallskip
\item For every $M \in \mathbb N_0$ and every finite nonempty set $\Delta \subset \mathbb N$, there is a finite abelian group $G$ such that the following holds{\rm \,:} for every {\rm AAMP} $L$ with difference $d\in \Delta$ and bound $M$ there is some $y_{L} \in \mathbb N$ such that
      \[
      y+L \in \mathcal L (G) \quad \text{ for all } \quad y \ge y_{L} \,.
      \]

\smallskip
\item If $G$ is an infinite abelian group, then
      \[
      \mathcal L (G) = \{ L \subset \N_{\ge 2} \mid L \ \text{is finite and
      nonempty} \
      \} \ \cup \ \{ \{0\}, \{1\} \} .
      \]
\end{enumerate}
\end{proposition}

\begin{proof}
1. Let $H$ be a transfer Krull monoid and $\theta \colon H \to \mathcal B (G_0)$ be a weak transfer homomorphism where $G_0$ is a finite subset of an abelian group. Then $\mathcal L (H) = \mathcal L (G_0)$, and $\mathcal B (G_0)$ satisfies the Structure Theorem by \cite[Theorem 4.4.11]{Ge-HK06a}.

\smallskip
For 2. we refer to   \cite{Sc09a}, and for 3. see \cite{Ka99a} and \cite[Section 7.4]{Ge-HK06a}.
\end{proof}

The inequalities in \eqref{inequality-Davenport} and the subsequent remarks show that a finite abelian group $G$ has Davenport constant $\mathsf D (G) \le 4$ if and only if $G$ is cyclic of order $|G|\le 4$ or if it is isomorphic to $C_2 \oplus C_2$ or to $C_2^3$. For these groups an explicit description of their systems of sets of lengths has been given, and we gather this in the next proposition (in Section \ref{4} we will determine the systems $\mathcal L (G)$ for all groups $G$ with $\mathsf D (G)=5$).

\medskip
\begin{proposition} \label{3.3}~

\begin{enumerate}
\item If $G$ is an abelian group, then $\mathcal L(G) = \{y +L \mid y \in \N_0, \ L \in \mathcal
      L(G^\bullet) \} \supset \bigl\{ \{y\} \,\bigm|\, y \in \N_0
      \bigr\}$, and equality holds if and only if \ $|G| \le 2$.

\smallskip
\item $\mathcal L (C_3) = \mathcal L (C_2 \oplus C_2) = \bigl\{ y
      + 2k + [0, k] \, \bigm| \, y,\, k \in \N_0 \bigr\}$.

\smallskip
\item $\mathcal L (C_4) = \bigl\{ y + k+1 + [0,k] \, \bigm|\, y,
      \,k \in \N_0 \bigr\} \,\cup\,  \bigl\{ y + 2k + 2 \cdot [0,k] \, \bigm|
      \, y,\, k \in \N_0 \bigr\} $.

\smallskip
\item $\mathcal L (C_2^3)  =  \bigl\{ y + (k+1) + [0,k] \,
      \bigm|\, y
      \in \N_0, \ k \in [0,2] \bigr\}$ \newline
      $\quad \text{\, } \ \qquad$ \quad $\cup \ \bigl\{ y + k + [0,k] \, \bigm|\, y \in \N_0, \ k \ge 3 \bigr\}
      \cup \bigl\{ y + 2k
      + 2 \cdot [0,k] \, \bigm|\, y ,\, k \in \N_0 \bigr\}$.
\end{enumerate}
\end{proposition}

\begin{proof}
See \cite[Proposition 7.3.1 and Theorem 7.3.2]{Ge-HK06a}.
 \end{proof}

\medskip
Let $G$ and $G'$ be abelian groups. Then their monoids of zero-sum sequences $\mathcal B (G)$ and $\mathcal B (G')$ are isomorphic if and only if the groups $G$ and $G'$ are isomorphic (\cite[Corollary 2.5.7]{Ge-HK06a}). The standing conjecture states that the systems of sets of lengths $\mathcal L (G)$ and $\mathcal L (G')$ of finite groups coincide if and only $G$ and $G'$ are isomorphic (apart from the trivial cases listed in Proposition \ref{3.3}). Here is the precise formulation of the conjecture (it was first stated in \cite{Ge16c}).

\medskip
\begin{conjecture} \label{3.4}
Let $G$ be a finite abelian group with $\mathsf D (G) \ge 4$. If $G'$ is an abelian group with $\mathcal L (G) = \mathcal L (G')$, then $G$ and $G'$ are isomorphic.
\end{conjecture}

\smallskip
The conjecture holds true for groups $G$ having rank $\mathsf r (G) \le 2$, for groups of the form $G=C_n^r$ (if $r$ is small with respect to $n$), and others (\cite{Ge-Sc16a, Ge-Zh16b, Zh17a}). But it is far open in general, and the goal of this section is to develop new  viewpoints of looking at this conjecture.

Let $G$  be a finite abelian group with $\mathsf D (G) \ge 4$. If $G'$ is a finite abelian group with $\mathcal L (G)=\mathcal L (G')$, then \eqref{inequality-elasticity} shows that
\[
\begin{aligned}
 \mathsf D (G) & = \rho_2 (G)  = \sup \{\sup L \mid 2 \in L \in \mathcal L (G) \} \\
     & =  \sup \{\sup L \mid 2 \in L \in \mathcal L (G') \} = \rho_2 (G') = \mathsf D (G') \,.
\end{aligned}
\]
We see from the inequalities in \eqref{inequality-Davenport} that there are (up to isomorphism) only finitely many finite abelian groups $G'$ with given Davenport constant, and hence there are only finitely many finite abelian groups $G'$ with $\mathcal L (G)=\mathcal L (G')$.
 Thus Conjecture \ref{3.4} is equivalent to the statement that for each $m \ge 4$ and for each two non-isomorphic finite abelian groups $G$ and $G'$ having Davenport constant $\mathsf D (G)=\mathsf D(G')=m$ the systems  $\mathcal L (G)$ and $\mathcal L (G')$ are distinct. Therefore we have to study the set
\[
\Omega_m = \{ \mathcal L (G) \mid G \ \text{is a finite abelian group with} \ \mathsf D (G)=m\}
\]
of all systems of sets of lengths stemming from groups having Davenport constant equal to $m$. If a group $G'$ is a proper subgroup of $G$, then $\mathsf D (G') < \mathsf D (G)$ (\cite[Proposition 5.1.11]{Ge-HK06a}) and hence  $\mathcal L (G') \subsetneq \mathcal L (G)$. Thus if $\mathsf D (G)=\mathsf D (G')$ for some group $G'$, then none of the groups is isomorphic to a proper subgroup of the other one. Conversely, if  $G'$ is a  finite abelian group with $\mathcal L (G') \subset \mathcal L (G)$, then $\mathsf D (G') = \rho_2 (G') \le \rho_2 (G) = \mathsf D (G)$. However, it may happen that $\mathcal L (G') \subsetneq \mathcal L (G)$ but $\mathsf D (G')= \mathsf D (G)$. Indeed, Proposition \ref{3.3} shows that $\mathcal L (C_4) \subsetneq \mathcal L (C_2^3)$, and we will observe  this phenomenon again in Section \ref{4}.

\medskip
\begin{theorem} \label{3.5}
For  $m \in \N$, let
\[
\Omega_m = \{ \mathcal L (G) \mid G \ \text{is a finite abelian group with} \ \mathsf D (G)=m\}.
\]
Then
$\mathcal L (C_2^{m-1})$ is a maximal element and $\mathcal L (C_m)$ is a minimal element in $\Omega_m$ (with respect to set-theoretical inclusion).
Furthermore, if $G$ is an abelian group with $\mathsf D (G)=m$ and $\mathcal L (G) \subset \mathcal L (C_2^{m-1})$, then $G \cong C_m $ or $G \cong C_2^{m-1}$.
\end{theorem}

\begin{proof}
If $m \in [1,2]$, then $|\Omega_m|=1$ and hence all assertions hold. Since $C_3$ and $C_2 \oplus C_2$ are the only groups (up to isomorphism) with Davenport constant three, and since $\mathcal L (C_3)= \mathcal L (C_2^2)$ by  Proposition \ref{3.3}, the assertions follow. We suppose that $m \ge 4$ and proceed in two steps.

1.
To show that $\mathcal L (C_2^{m-1})$ is maximal, we study, for a finite abelian group $G$, the set $\Delta_1 (G)$. We define
\[
\Delta^* (G) = \{ \min \Delta (G_0) \mid G_0 \subset G \ \text{with} \ \Delta (G_0) \ne \emptyset \} \,,
\]
and recall that (see \cite[Corollary 4.3.16]{Ge-HK06a})
\[
\Delta^* (G) \subset \Delta_1 (G) \subset \{ d_1 \in \Delta (G) \mid d_1 \ \text{divides some} \ d \in \Delta^* (G) \} \,.
\]
Thus $\max \Delta_1 (G) = \max \Delta^* (G)$, and  \cite[Theorem 1.1]{Ge-Zh16a} implies that $\max \Delta^* (G) = \max \{\exp (G)-2, \mathsf r (G)-1\}$.
Assume to the contrary that there is  a finite abelian group $G$ with $\mathsf D (G) =m\ge4$ that is not an elementary $2$-group such that $\mathcal L (C_2^{m-1} ) \subset \mathcal L (G)$. Then
\[
\begin{split}
m-2
& =\max \Delta^*(C_2^{m-1}) = \max \Delta_1 (C_2^{m-1}) \le \max \Delta_1 (G) \\
&  =\max \Delta^* (G) = \max \{ \exp (G)-2, \mathsf r (G)-1\} \,.
\end{split}
\]
If $\mathsf r (G) \ge m-1$, then $\mathsf D (G)=m$ implies that $G \cong C_2^{m-1}$, a contradiction. Thus $\exp (G) \ge m$, and since $\mathsf D (G)=m$ we infer that that $G \cong C_m$. If $m=4$, then Proposition \ref{3.3}.4 shows that $\mathcal L (C_2^3) \not\subset \mathcal L (C_4)$, a contradiction. Suppose that $m \ge 5$. Then $\Delta^*(C_2^{m-1}) = \Delta_1 (C_2^{m-1}) = \Delta (C_2^{m-1}) = [1, m-2]$ by \cite[Corollary 6.8.3]{Ge-HK06a}. For cyclic groups we have $\max \Delta^* (C_m) = m-2$ and
$\max ( \Delta^* (C_m) \setminus \{m-2\}) = \lfloor m/2 \rfloor - 1$ by \cite[Theorem 6.8.12]{Ge-HK06a}.
Therefore $\mathcal L (C_2^{m-1} ) \subset \mathcal L (C_m)$ implies that
\[
[1, m-2] = \Delta_1 (C_2^{m-1}) \subset \Delta_1 (C_m) \,,
\]
a contradiction to $m-3 \notin \Delta_1 (C_m)$.

\smallskip
2. We recall some facts. Let $G$ be a group with $\mathsf D (G)=m$. If $U \in \mathcal A (G)$ with $|U|=\mathsf D (G)$, then $\{2, \mathsf D (G) \} \subset \mathsf L \big( U(-U) \big)$. Cyclic groups and elementary $2$-groups are the only groups $G$ with the following property: if $L \in \mathsf L (G)$ with $\{2, \mathsf D (G)\} \subset L$, then $L = \{2, \mathsf D (G)\}$ (\cite[Theorem 6.6.3]{Ge-HK06a}).

Now assume to the contrary that there is a finite abelian group $G$ with $\mathsf D (G)=m$ such that $\mathcal L (G) \subset \mathcal L (C_m)$. Let $L  \in \mathcal L (G)$ with $\{2, \mathsf D (G) \} \subset  L$. Then $L \in \mathcal L (C_m)$ whence $L = \{2, \mathsf D (G)\}$ which implies that $G$ is cyclic or an elementary $2$-group. By 1., $G$ is not an elementary $2$-group whence $G$ is cyclic which implies $G \cong C_m$ and hence $\mathcal L (G) = \mathcal L (C_m)$.

The furthermore assertion on groups $G$ with $\mathsf D (G)=m$ and $\mathcal L (G) \subset \mathcal L (C_2^{m-1})$ follows as above by considering sets of lengths $L$ with $\{2, \mathsf D (G)\} \subset L$.
 \end{proof}

In Section \ref{4} we will see that $\mathcal L (C_2^{m-1})$ need not be the largest element in $\Omega_m$, and that indeed $\mathcal L (C_m) \subset\mathcal L (C_2^{m-1})$ for $m \in [2,5]$, where the inclusion is strict for $m \ge 4$. On the other hand, it is shown in \cite{Ge-Zh18a} that $\mathcal L (C_m) \not\subset\mathcal L (C_2^{m-1})$ for infinitely many $m \in \N$.

\medskip
\begin{theorem} \label{3.6}
We have
 \[
      \bigcap \mathcal L (G) = \bigl\{ y
      + 2k + [0, k] \, \bigm| \, y,\, k \in \N_0 \bigr\} \,,
      \]
      where the intersection is taken over all finite abelian groups $G$ with $|G| \ge 3$.
\end{theorem}

\begin{proof}
By Proposition \ref{3.3}.2, the intersection on the left hand side is contained in the set on the right hand side. Let $G$ be a finite abelian group with $|G| \ge 3$. If $L \in \mathcal L (G)$, then $y+L \in \mathcal L (G)$. Thus it is sufficient to show that $[2k,3k] \in \mathcal L (G)$ for every $k \in \N$.
If $G$ contains two independent elements of order $2$ or an element of order $4$, then the claim follows by Proposition \ref{3.3}. Thus, it remains to consider the case when $G$ contains an element $g$ with $\ord (g)= p$ for some odd prime $p \in \N$. Let $k \in \N$ and  $B_k = ((2g)^p g^p)^k$. We assert that $\mathsf L (B_k)=[2k,3k]$.

We set $U_1= g^p$, $U_2 = (2g)^p$, $V_1 = (2g)^{(p-1)/2}g$, and $V_2=(2g)g^{p-2}$. Since $U_1U_2=V_1^2V_2$ and
\[
B_k = (U_1U_2)^k = (U_1 U_2)^{k-\nu} (V_1^2V_2)^{\nu} \quad \text{for all} \quad \nu \in [0,k] \,,
\]
it follows that $[2k,3k] \subset \mathsf L (B_k)$.

In order  to show there are no other factorization lengths, we recall the concept of the $g$-norm of sequences. If $S = (n_1g) \cdot \ldots (n_{\ell}g) \in \mathcal B (\langle g \rangle)$, where $\ell \in \N_0$ and $n_1, \ldots, n_{\ell} \in [1, \ord (g)]$, then
\[
||S||_g = \frac{n_1+ \ldots + n_{\ell}}{\ord (g)} \in \N
\]
is the $g$-norm of $S$. Clearly, if $S = S_1 \cdot \ldots \cdot S_m$ with $S_1, \ldots , S_m \in \mathcal A (G)$, then $||S||_g = ||S_1||_g + \ldots + ||S_m||_g$.

Note that $U_2=(2g)^p$ is the only atom in $\mathcal A (\{g,2g\})$ with $g$-norm $2$, and all other atoms in $\mathcal A (\{g,2g\})$ have $g$-norm $1$.
Let $B_k = W_1 \cdot \ldots  \cdot W_{\ell}$ be a factorization of $B_k$, and
let $\ell'$ be the number of $i \in[1,\ell] $ such that $W_i = (2g)^p$.
We have $\| B_k\|_g= 3k$ and  thus $3k = 2\ell' + (\ell- \ell')= \ell' + \ell$. Since  $\ell' \in [0,k]$, it follows that $\ell =3k-\ell' \in [2k , 3k] $.
 \end{proof}

\medskip
\begin{theorem} \label{3.7}
Let $L \subset \N_{\ge 2}$ be a finite  nonempty subset. Then there are only finitely many pairwise non-isomorphic finite abelian groups $G$  such that $L \notin \mathcal L (G)$.
\end{theorem}

\begin{proof}
We start with the following two assertions.

\smallskip
\begin{enumerate}
\item[{\bf A1.}\,] There is an integer $n_L \in \N$ such that $L \in \mathcal L (C_n)$ for every $n \ge n_L$.

\item[{\bf A2.}\,] For every $p \in \P$ there is an integer $r_{p,L} \in \N$ such that $L \in \mathcal L (C_p^r)$ for every $r \ge r_{p,L}$.
\end{enumerate}

\noindent
{\it Proof of} \,{\bf A1}.\,
By Proposition \ref{3.2}.3, there is some $B = \prod_{i=1}^k m_k \prod_{j=1}^{\ell} (-n_j) \in \mathcal B ( \Z)$ such that $\mathsf L (B)=L$, where $k,\ell, m_1, \ldots, m_k \in \N$ and $n_1, \ldots, n_{\ell} \in  \N_0$. We set $n_L = n_1+ \ldots + n_{\ell}$ and choose some $n \in \N$ with $n \ge n_L$. If $S \in \mathcal F (\Z)$ with $S \t B$ and $f \colon \Z \to \Z/n\Z$ denotes the canonical epimorphism, then $S$ has sum zero if and only if $f (S)$ has sum zero. This implies that $\mathsf L_{\mathcal B (\Z/n\Z)} ( f (B)) = \mathsf L_{\mathcal B ( \Z)} (B) = L$.
\qed[Proof of {\bf A1}]

\medskip
\noindent
{\it Proof of} \,{\bf A2}.\,
Let $p \in \P$ be a prime and let $G_p$ be an infinite dimensional $\F_p$-vector space. By Proposition \ref{3.2}.3, there is some $B_p \in \mathcal B (G_p)$ such that $\mathsf L (B_p)=L$. If $r_{p,L}$ is the rank of $\langle \supp (B_p) \rangle \subset G_p$, then
\[
L = \mathsf L (B_p) \in \mathcal L ( \langle \supp (B_p) \rangle ) \subset \mathcal L (C_p^r) \quad \text{for } \quad r \ge r_{p,L} \,.  \qquad \qed[\text{Proof of} {\bf A2}]
\]

\medskip
Now let $G$ be a finite abelian group such that $L \notin \mathcal L (G)$. Then {\bf A1} implies that $\exp (G) < n_L$, and {\bf A2} implies that $\mathsf r_p (G) < r_{p,L}$ for all primes $p$ with $p \t \exp (G)$. Thus the assertion follows.
 \end{proof}

\medskip
\section{Sets of lengths of transfer Krull monoids over small groups} \label{4}
\medskip

Since the very beginning of factorization theory, invariants controlling the structure of sets of lengths (such as elasticities and sets of distances) have been in the center of interest. Nevertheless, (apart from a couple of trivial cases) the full system of sets of lengths has been written down explicitly only for the following classes of monoids:
\begin{itemize}
\item Numerical monoids generated by arithmetical progressions: see \cite{A-C-H-P07a}.

\item Self-idealizations of principal ideal domains: see \cite[Corollary 4.16]{Ch-Sm13a}, \cite[Remark 4.6]{Ba-Ba-Mc15a}.

\item The ring of integer-valued polynomials over $\Z$: see    \cite{Fr13a}.

\item The systems $\mathcal L (G)$ for infinite abelian groups $G$ and for abelian groups $G$ with $\mathsf D (G) \le 4$: see Propositions \ref{3.2} and \ref{3.3}.
\end{itemize}
The goal of this section is to determine $\mathcal L (G)$ for abelian groups $G$ having Davenport constant $\mathsf D (G) = 5$.
By inequality \eqref{inequality-Davenport} and the subsequent remarks, a finite abelian group $G$ has Davenport constant five if and only if it is isomorphic to one of the following groups:
\[
C_3 \oplus C_3,  \quad C_5, \quad  C_2 \oplus C_4, \quad   C_2^4 \,.
\]
Their systems of sets of lengths are  given in Theorems \ref{4.1}, \ref{4.6}, \ref{4.7}, and \ref{4.8}. We start with a brief analysis of these explicit descriptions (note that they will be needed again in Section \ref{5}; confer the proof of Theorem \ref{5.7}).

By Theorem  \ref{3.5}, we know that $\mathcal L (C_2^4)$ is maximal in $\Omega_5 = \{ \mathcal L (C_5), \mathcal L (C_2 \oplus C_4), \mathcal L (C_3\oplus C_3), \mathcal L (C_2^4) \}$.  Theorems \ref{4.1}, \ref{4.6}, \ref{4.7}, and \ref{4.8} unveil that $\mathcal L (C_3 \oplus C_3)$, $\mathcal L (C_2\oplus C_4)$, and $\mathcal L (C_2^4)$ are maximal in $\Omega_5$, and that $\mathcal L (C_5)$ is contained in $\mathcal L (C_2^4)$, but it is neither contained in $\mathcal L (C_3 \oplus C_3)$ nor in $\mathcal L (C_2 \oplus C_4)$. Furthermore, Theorems \ref{3.5}, \ref{4.6}, and \ref{4.8}  show that $\mathcal L (C_m) \subset \mathcal L (C_2^{m-1})$ for $m \in [2,5]$. It is well-known that, for all $m \ge 4$, $\mathcal L (C_m) \ne \mathcal L (C_2^{m-1})$ (\cite[Corollary 5.3.3]{Ge09a}), and the standing conjecture is that $\mathcal L (C_m) \not\subset \mathcal L (C_2^{m-1})$ holds true for almost  all $m \in \N_{\ge 2}$ (see \cite{Ge-Zh18a}).

\smallskip
The group $C_3 \oplus C_3$ has been handled in \cite[Theorem 4.2]{Ge-Sc16a}.

\medskip
\begin{theorem} \label{4.1}
$\mathcal L (C_3^2) = \{ y+ [2k, 5k] \mid y,k \in \mathbb N_0\} \ \cup \ \{ y +  [2k+1, 5k+2] \mid y \in \N_0,\, k \in \N \} \}$.
\end{theorem}

\noindent
{\it Remark.} An equivalent way to describe $\mathcal L (C_3^2)$ is $\{y+\left\lceil\frac{2k}{3}\right\rceil+[0,k]\mid y\in \N_0, k\in \N_{\ge 2} \}\cup\{ \{y\}, y+2+[0,1]\mid y\in \N_0\}$.

The fact that all sets of lengths are intervals is a consequence of the fact $\Delta(C_3^2)=\{1\}$. Of course, each set of lengths $L$ has to fulfill $\rho(L) \le 5/2 = \rho(C_3^2)$. We observe that the description shows that this is the only condition, provided $\min L \ge 2$.
The following lemma is frequently helpful in the remainder of this section.

\medskip
\begin{lemma}\label{4.2}
Let $G$ be a finite abelian group, and let $A\in \mathcal B(G)$.
\begin{enumerate}
\item If $\supp(A)\cup\{0\}$ is a group, then $\mathsf L(A)$ is an interval.

\smallskip
\item If $A_1$ is an atom dividing $A$ with $|A_1|=2$, then $\max \mathsf L(A)=1+\max \mathsf L(AA_1^{-1})$.

\smallskip
\item If $A$ is a product of atoms of length $2$ and if every atom $A_1$ dividing $A$ has length $|A_1|=2$ or $|A_1|=4$, then $\max\mathsf L(A)-1\notin \mathsf L(A)$.
\end{enumerate}

\end{lemma}

\begin{proof}
1. See  \cite[Theorem 7.6.8]{Ge-HK06a}.

\smallskip
2. Let $\ell=\max\mathsf L(A)$ and $A=U_1\cdot\ldots\cdot U_{\ell}$, where $U_1,\ldots, U_{\ell}\in \mathcal A(G)$. Let $A_1=g_1g_2$, where $g_1,g_2\in G$. If there exists $i\in [1,\ell]$
 such that $A_1=U_i$, then $\max \mathsf L(A)=1+\max \mathsf L(AA_1^{-1})$. Otherwise there exist distinct $i,j\in [1,\ell]$ such that $g_1\t U_i$ and $g_2\t U_j$. Thus $A_1$ divides $U_iU_j$ and hence $1+\max \mathsf L(AA_1^{-1})\ge \ell$ which implies that $\max \mathsf L(A)=1+\max \mathsf L(AA_1^{-1})$ by the maximality of $\ell$.

\smallskip
 3. If $\max\mathsf L(A)-1\in \mathsf L(A)$, then $A=V_1\cdot\ldots \cdot V_{\max\mathsf L(A)-1}$ with $|V_1|=4$ and $|V_2|=\ldots =|V_{\max\mathsf L(A)-1}|=2$. Thus $V_1$ can only be a product two atoms of length $2$, a contradiction.
  \end{proof}

We now consider the  groups $C_5$, $C_2 \oplus C_4$, and $C_2^4$, each one  in its own subsection.
In the proofs of the forthcoming theorems we will use Proposition \ref{3.3} and Theorem \ref{3.6} without further mention.

\subsection{The system of sets of lengths of $C_5$}

The goal of this subsection is to prove the following result.

\medskip
\begin{theorem} \label{4.6}
$\quad \mathcal L(C_5)=\mathcal L_1\cup \mathcal L_2\cup \mathcal L_3\cup \mathcal L_4\cup \mathcal L_5\cup \mathcal L_6\,,$ where
\begin{align*}
& \mathcal L_1=\{\{y\}\mid y\in \N_0\} \,, \\
&\mathcal L_2=\{y+2+\{0,2\}\mid y\in \N_0\} \,, \\
&\mathcal L_3= \{y+3+\{0,1,3\}\mid y\in \N_0\}\,, \\
&\mathcal L_4=\{y+2k+3 \cdot [0,k]\mid y\in \N_0, k\in \N\}\,,\\
&\mathcal L_5=\{y+2\left\lceil\frac{k}{3}\right\rceil+[0,k]\mid y\in \N_0, k\in \N\setminus\{3\}\}\cup\{y+[3,6]\mid y\in \N_0\}\, , \text{ and } \\
& \mathcal L_6= \{y+2k+3+\{0,2,3\}+3 \cdot [0,k]\mid y, k \in \N_0 \}\,.
\end{align*}
\end{theorem}

We observe that all sets of lengths with many elements are arithmetic multiprogressions with difference $1$ or $3$. Yet, there are none with difference $2$. This is because $\Delta^{*}(C_5) = \{1,3\}$.
Moreover, we point out that the condition for an interval to be a set of lengths is different from that of the other groups with Davenport constant $5$. This is related to the fact that $\rho_{2k+1}(C_5) = 5k+1$, while $\rho_{2k+1}(G) = 5k+2$ for the other groups with  Davenport constant $5$.
Before we start the actual proof, we collect some results on sets of lengths over $C_5$.

\medskip
\begin{lemma}\label{4.3}
Let $G$ be cyclic of order five,  and let $A \in \mathcal B(G)$.
\begin{enumerate}
\item  If $g \in G^{\bullet}$ and $k \in \N_0$, then
\[
\mathsf L \big (g^{5(k+1)} (-g)^{5(k+1)} (2g)g^3 \big)=2k+3+\{0,2,3\}+3 \cdot [0,k] \,.
\]

\smallskip
\item If $2\in \Delta (\mathsf L (A))\subset [1,2]$, then $\mathsf L(A)\in  \{\{y,y+2\}\mid y\ge 2\}\cup \{\{y,y+1,y+3\}\mid y\ge 3\}$ or $\mathsf L(A)=3+\{0,2,3\}+\mathsf L(A')$ where $A'\in \mathcal B(G)$ and $\mathsf L(A')$ is an arithmetical progression of difference $3$.

\smallskip
\item $\Delta(G)=[1,3]$, and  if \ $3\in \Delta (\mathsf L (A))$,  then $\Delta (\mathsf L (A))=\{3\}$.

\smallskip
\item $\rho_{2k+1} (G) = 5k+1$ for all $k \in \N$.
\end{enumerate}

\end{lemma}

\begin{proof}
1. and 2. follow from  the proof of \cite[Lemma 4.5]{Ge-Sc16a}.

3. See \cite[Theorems 6.7.1 and 6.4.7]{Ge-HK06a} and \cite[Theorem 3.3]{Ch-Go-Pe14a}.

4. See \cite[Theorem 5.3.1]{Ge09a}.
 \end{proof}

\begin{proof}[Theorem \ref{4.6}]
Let $G$ be cyclic of order five and let $g\in G^{\bullet}$.
We first show that all the specified sets occur as sets of lengths, and then we show that no other sets occur.

\smallskip
\noindent
{\bf Step 1.} We prove that for every $L\in  \mathcal L_2\cup \mathcal L_3\cup \mathcal L_4\cup \mathcal L_5\cup \mathcal L_6$, there exists an $A\in \mathcal B(G)$ such that $L=\mathsf L(A)$. We distinguish five cases.

\smallskip

If  $L=\{y,y+2\}\in \mathcal L_2$ with $y\ge 2$, then we set $A=0^{y-2} g^5 (-g)^3(-2g)$ and obtain that $\mathsf L(A)=y-2+\{2,4\}=L$.

If $L=\{y,y+1,y+3\}\in \mathcal L_3$ with $y\ge 3$, then we set $A=0^{y-3}  g^5 (-g)^5 g^2(-2g)$ and obtain that $\mathsf L(A)=y-3+\{3,4,6\}=\{y,y+1,y+3\}=L$.

If $L=y+2k+3 \cdot [0,k]\in \mathcal L_4$ with $k\in \N$ and $y\in \N_0$, then we set $A=g^{5k}(-g)^{5k}0^y\in \mathcal B(G)$ and hence  $\mathsf L(A)=y+2k+3 \cdot [0,k]=L$.

If $L=y+2k+3+\{0,2,3\}+3 \cdot [0,k]\in \mathcal L_6$ with $k\in \N_0$ and $y\in \N_0$, then we set $A=0^y g^{5(k+1)} (-g)^{5(k+1)} (2g)g^3$ and hence  $\mathsf L(A)=y+2k+3+\{0,2,3\}+3 \cdot [0,k]=L$ by Lemma \ref{4.3}.1.

Now we suppose that $L\in \mathcal L_5$, and we distinguish two subcases. First, if $L=y+[3,6]$ with $y\in \N_0$, then we set $A=0^y (2g(-2g)) g^{5} (-g)^{5}$ and hence $\mathsf L(A)=y+[3, 6]=L$. Second, we assume that $L=y+2\lceil\frac{k}{3}\rceil+[0,k]$ with $y\in \N_0$ and $k\in \N\setminus\{3\}$.

If $k\in \N$ with $k\equiv 0\pmod 3$, then $k\ge 6$ and  by Lemma \ref{4.2}.1 we obtain that
\begin{align*}
\mathsf L \big( 0^y  (2g)^5 (-2g)^5  g^{5t}  (-g)^{5t} \big) = y+[2t+2, 5t+5]=y+2\lceil\frac{k}{3}\rceil+[0,k]=L\,,
\end{align*}
where $k=3t+3$.

If $k\in \N$ with $k\equiv 1\pmod 3$, then  by Lemma \ref{4.2}.1 we obtain that
\[
\mathsf L \big(0^y (2g(-g)^2) (g^2(-2g)) g^{5t} (-g)^{5t}\big) = y+[2t+2, 5t+3]=y+2\lceil\frac{k}{3}\rceil+[0,k]=L\,,
\]
where $k=3t+1$.

If $k\in \N$ with $k\equiv 2\pmod 3$, then  by Lemma \ref{4.2}.1 we obtain that
\[
\mathsf L\big(0^y  (g^3(2g))  ((-g)^3(-2g))  g^{5t}  (-g)^{5t}\big) = y+[2t+2, 5t+4]=y+2\lceil\frac{k}{3}\rceil+[0,k]=L\,,
\]
 where $k=3t+2$.

\medskip
\noindent
{\bf Step 2.} We prove that for every $A\in \mathcal B(G^{\bullet})$, $\mathsf L(A)\in  \mathcal L_2\cup \mathcal L_3\cup \mathcal L_4\cup \mathcal L_5\cup \mathcal L_6$.

Let $A \in \mathcal B (G^{\bullet})$. We may suppose that $\Delta (\mathsf L (A)) \ne \emptyset$. By Lemma \ref{4.3}.3 we distinguish three cases according to the form of the set of distances $\Delta (\mathsf L (A))$.

\smallskip
\noindent
CASE 1: \,  $\Delta (\mathsf L (A))=\{1\}$.

Then $\mathsf L(A)$ is an interval and hence we assume that $\mathsf L(A)=[y,y+k]=y+[0,k]$ where $y\ge 2$ and $k\ge 1$. If $k=3$ and $y=2$, then $\mathsf L(A)=[2,5]$ and hence $\mathsf L(A)=\mathsf L(g^5(-g)^5)=\{2,5\}$, a contradiction. Thus $k=3$ implies that $y\ge 3$ and hence $\mathsf L(A)\in \mathcal L_5$. If $k\le 2$, then we obviously have that $\mathsf L(A)\in \mathcal L_5$. Suppose  that $k\ge 4$.
If $y=2t$ with $t\ge 2$, then $y+k\le 5t$ and hence $y=2t\ge 2\lceil\frac{k}{3}\rceil$ which implies that $\mathsf L(A)\in \mathcal L_5$. If $y=2t+1$ with $t\ge 1$, then $y+k\le 5t+1$ and hence $y=2t+1\ge 1+ 2\lceil\frac{k}{3}\rceil$ which implies that $\mathsf L(A)\in \mathcal L_5$.

\smallskip
\noindent
CASE 2: \, $\Delta (\mathsf L (A))=\{3\}$.

Then $\mathsf L(A)=y+3 \cdot [0,k]$ where $y\ge 2$ and $k\ge 1$. If $y=2t\ge 2$, then $y+3k\le 5t$ and hence $y=2t\ge 2k$ which implies that $\mathsf L(A)\in \mathcal L_4$. If $y=2t+1\ge 3$, then $y+3k\le 5t+1$ and hence $y=2t+1\ge 1+ 2k$ which implies that $\mathsf L(A)\in \mathcal L_4$.

\smallskip
\noindent
CASE 3: \, $2\in \Delta (\mathsf L (A))\subset [1,2]$.

By Lemma \ref{4.3}.2, we infer that either $\mathsf L(A)\in \mathcal L_2\cup \mathcal L_3$ or that $\mathsf L(A)=3+\{0,2,3\}+\mathsf L(A')$, where $A'\in \mathcal B(G)$ and $\mathsf L(A')$ is an arithmetical progression of difference $3$. In  the latter case  we obtain that $\mathsf L(A')=y+2k+3 \cdot [0,k]$, with $y\in \N_0$ and $k\in \N_0$, and hence $\mathsf L(A)=y+2k+3+\{0,2,3\}+3 \cdot [0,k]\in \mathcal L_6$.
 \end{proof}

\subsection{The system of sets of lengths of $C_2 \oplus C_4$}

We establish the following result, giving a complete description of the system of sets of lengths of $C_2 \oplus C_4$.

\medskip
\begin{theorem} \label{4.7}
$\quad \mathcal L(C_2\oplus C_4)=\mathcal L_1\cup \mathcal L_2\cup \mathcal L_3\cup \mathcal L_4\cup \mathcal L_5\,,
$ where
\begin{align*}
\mathcal L_1&=\{\{y\}\mid y\in \N_0 \}\,,\\
\mathcal L_2&=\{y+2\left\lceil\frac{k}{3}\right\rceil+[0,k]\mid y\in \N_0, k\in \N\setminus\{3\} \}\cup \\   &  \quad \quad \{y+[3,6]\mid y\in \N_0,\}\cup\{[2t+1, 5t+2]\mid t\in \N\} \\
&=\{y+\left\lceil\frac{2k}{3}\right\rceil+[0,k]\mid y\in \N_0, k\in \N\setminus\{1,3\} \}\cup \\& \quad \quad \{y+3+[0,3], y+2+[0,1]\mid y\in \N_0\}\,,\\
\mathcal L_3&=\{y+2k+2 \cdot [0,k]\mid y\in \N_0, k\in \N\}\,,\\
\mathcal L_4&= \{y+k+1+(\{0\}\cup[2,k+2])\mid y\in \N_0, k\in \N \text{\rm \  odd}\}\,, \text{ and} \\
\mathcal L_5&= \{y+k+2+([0,k]\cup \{k+2\})\mid y\in \N_0, k\in \N\}\,. \\
\end{align*}
\end{theorem}

We note that all sets of lengths are arithmetical progressions with difference $2$ or almost arithmetical progressions with difference $1$ and bound $2$.
This is related to the fact that $\Delta ( C_2 \oplus C_4 )= \Delta^{*} ( C_2 \oplus C_4 ) =\{1,2\}$.
We start with a lemma  determining all minimal zero-sum sequences over $C_2 \oplus C_4$.

\medskip
\begin{lemma} \label{lem_zssC2C4}
Let $(e,g)$ be a basis of $G=C_2\oplus C_4$ with $\ord(e)=2$ and $\ord(g)=4$.
Then the minimal zero-sum sequences over $G^{\bullet}$ are given by the following list.
\begin{enumerate}
\item The minimal zero sum sequences of  length $2$ are{\rm \,:}
\begin{align*}
&S_{2}^1=\{e^2, (e+2g)^2\},\\
 &S_2^2=\{(2g)^2\},\\
  &S_2^3=\{g(-g), (e+g)(e-g)\}
\end{align*}
\item The minimal zero sum sequences of  length $3$ are{\rm \,:}
\begin{align*}
&S_3^1=\{e(2g)(e+2g)\}\,,\\
& S_3^2=\{g^2(2g),(-g)^2(2g), (e+g)^2(2g), (e-g)^2(2g)\}\,,\\
 &S_3^3=\{eg(e-g), e(-g)(e+g), (e+2g)g(e+g), (e+2g)(-g)(e-g)\}\,.\\
\end{align*}

\item The minimal zero sum sequences of  length $4$ are{\rm \,:}
\begin{align*}
&S_4^1=\{g^4, (-g)^4, (e+g)^4, (e-g)^4\}\,,\\
 &S_4^2=\{g^2(e+g)^2, (-g)^2(e-g)^2, g^2(e-g)^2, (-g)^2(e+g)^2\}\,,\\
 & S_4^3=\{eg^2(e+2g),e(e+g)^2(e+2g), e(-g)^2(e+2g), e(e-g)^2(e+2g)\}\,,\\
  & S_4^4=\{eg(2g)(e+g),e(-g)(2g)(e-g), (e+2g)g(2g)(e-g), (e+2g)(-g)(2g)(e+g)\}\,.\\
\end{align*}

\item The minimal zero sum sequences of  length $5$ are{\rm \,:}
\begin{align*}
S_5=\{&eg^3(e+g), e(-g)^3(e-g), e(e+g)^3g, e(e-g)^3(-g)\\
&(e+2g)g^3(e-g), (e+2g)(-g)^3(e+g), \\ & (e+2g)(e+g)^3(-g), (e+2g)(e-g)^3g\}\,,\\
 \end{align*}
\end{enumerate}
Moreover,  for each two atoms $W_1,W_2$ in any one of the above sets, there exists a group isomorphism $\phi \colon G \to G$ such that $\phi(W_1)=W_2$.
\end{lemma}

\begin{proof}
We give a sketch of the proof.

Since a minimal zero-sum sequence of length two is of the form $h(-h)$ for some non-zero element $h\in G$, the list given in 1. follows.

A minimal zero-sum sequence of length three contains either two elements of order four or no element of order four.
If there are two elements of order four, we can have one element of order four with multiplicity two (see $S_3^2$) or two distinct elements of order four that are not the inverse of each other (see $S_3^3$).
If there is no element of order four,  the sequence consists of  three distinct elements of order two (see $S_3^1$).

A minimal zero-sum sequence  of length four contains either four elements of order four or two elements of order four.
If there are two elements of order four, the sequence can contain one element with multiplicity two (see $S_4^3$) or any two distinct elements that are not each other's inverse with multiplicity one (see $S_4^4$).
If there are four elements of order four, the sequence can contain  one element with multiplicity four (see $S_4^1$) or  two elements with multiplicity two (see $S_4^2$).

Since every minimal zero-sum sequence of length five contains an element with multiplicity three, the list given in 4. follows (for details see \cite[Theorem 6.6.5]{Ge-HK06a}).

The existence of the required isomorphism follows immediately from the given description of the sequences.
 \end{proof}

The next lemma collects some basic results on $\mathcal L (C_2 \oplus C_4)$ that will be essential for the proof of Theorem \ref{4.7}.

\medskip
\begin{lemma} \label{4.4}
Let $G=C_2\oplus C_4$, and let $A \in \mathcal B (G)$.
\begin{enumerate}
\item $\Delta(G)=[1,2]$, and if  $\{2,5\}\subset \mathsf L(A)$, then $\mathsf L(A)=\{2,4,5\}$.

\smallskip
\item $\rho_{2k+1} (G) = 5k+2$ for all $k \in \N$.

\smallskip
\item If $(e,g)$ is a basis of $G$ with $\ord(e)=2$ and $\ord(g)=4$, then $\{0,g,2g,e+g,e+2g\}$ and $\{0,g,2g,e,e-g\}$ are half-factorial sets. Furthermore, if $\supp(A)\subset\{e,g,2g,e+g,e+2g\}$ and $\mathsf v_e(A)=1$, then $|\mathsf L(A)|=1$.
\end{enumerate}

\end{lemma}

\begin{proof}
1.   The first assertion follows from  \cite[Theorem 6.7.1 and Corollary 6.4.8]{Ge-HK06a}. Let $A \in \mathcal B (G)$ with $\{2,5\}\subset \mathsf L(A)$. Then there is an $U \in \mathcal A (G)$ of length $|U|=5$ such that $A = (-U)U$. By Lemma \ref{lem_zssC2C4} there is a basis $(e,g)$ of $G$ with $\ord (e)=2$ and $\ord (g)=4$ such that $U = eg^3(e+g)$. This implies that   $\mathsf L(A)=\{2,4,5\}$.

\smallskip
2. See \cite[Corollary 5.2]{Ge-Gr-Yu15}.

\smallskip
3. See \cite[Theorem 6.7.9.1]{Ge-HK06a}  for the first statement. Suppose that $\supp(A)\subset\{e,g,2g,e+g,e+2g\}$ and $\mathsf v_e(A)=1$. Then for every atom $W$ dividing $A$ with $e\t W$, we have that $\mathsf k(W)=\frac{3}{2}$. Since $\supp(AW^{-1})$ is half-factorial, we obtain that $\mathsf L(AW^{-1})=\{\mathsf k(A)-3/2\}$ by \cite[Proposition 6.7.3]{Ge-HK06a} which implies that $\mathsf L(A)=\{1+\mathsf k(A)-3/2\}=\{\mathsf k(A)-1/2\}$.
 \end{proof}

\begin{proof}[Theorem \ref{4.7}]
Let $(e,g)$ be a basis of $G=C_2\oplus C_4$ with $\ord(e)=2$ and $\ord(g)=4$.
We start by collecting some basic constructions that will be useful. Then, we show that all the sets in the result actually are sets of lengths.
Finally, we show there are no other sets of lengths.

\medskip
\noindent
{\bf Step 0.}  Some elementary constructions.

 Let $U_1=eg^3(e+g)$, $U_2=(e+2g)(e+g)^3(-g)$, $U_3=e(e-g)^3(-g)$,  $U_4=(-g)^2(e+g)^2$, and $U_5=e(e+2g)g^2$.
 Then it is not hard to check that
\begin{align}
&\mathsf L(U_1(-U_1))=\mathsf L(U_2(-U_2))=\{2,4,5\},&&\nonumber\\
&\mathsf L(U_1U_3))=[2,4],  \quad &&\mathsf L(U_1(-U_4))=[2,3]\,,\nonumber\\
&\mathsf L(U_1U_3U_4)=[3,7], \quad &&\mathsf L(U_1(-U_1)U_2(-U_2))=[4,10]\,,\nonumber\\
&\mathsf L(U_5^2(-g)^4)=\{3,4,6\}, \quad &&\mathsf L(U_5(-U_5)g^4(-g)^4)=\{4,5,6,8\}\,,\nonumber \\
 \label{eq1}&\mathsf L(U_1(-U_1)(e+2g)^2)=[3,6]\,.&&
\end{align}
Based on these results, we can obtain the sets of lengths of more complex zero-sum sequences. Let $k \in \N$.

Since $[2k+2,4k+5]\supset \mathsf L(U_1(-U_1)g^{4k}(-g)^{4k})\supset \mathsf L(U_1(-U_1))+\mathsf L(g^{4k}(-g)^{4k})=2k+2+(\{0\}\cup[2,2k+3])$ and $2k+3\notin \mathsf L(U_1(-U_1)g^{4k}(-g)^{4k})$, we obtain that \begin{equation}
\label{eq2}\mathsf L(U_1(-U_1)g^{4k}(-g)^{4k})=2k+2+(\{0\}\cup[2,2k+3])\,.
\end{equation}

Since $[2(k+1), 5(k+1)]\supset \mathsf L(U_1(-U_1)U_2^k(-U_2)^k)\supset \mathsf L(U_1(-U_1)U_2(-U_2))+\mathsf L(U_2^{k-1}(-U_2)^{k-1})=[2(k+1), 5(k+1)]$, we obtain that \begin{equation}
\label{eq3} \mathsf L(U_1(-U_1)U_2^k(-U_2)^k)=[2(k+1), 5(k+1)]\,.\\
\end{equation}

Since $[2(k+1), 5(k+1)-1]\supset \mathsf L(U_1U_3U_2^k(-U_2)^k)\supset \mathsf L(U_1U_3)+\mathsf L(U_2^{k}(-U_2)^{k})=[2(k+1), 5(k+1)-1]$, we obtain that
\begin{equation}
\label{eq4} \mathsf L(U_1U_3U_2^k(-U_2)^k)=[2(k+1), 5(k+1)-1]\,.\\
\end{equation}

Since $[2(k+1), 5(k+1)-2]\supset \mathsf L(U_1(-U_4)U_2^k(-U_2)^k)\supset \mathsf L(U_1(-U_4))+\mathsf L(U_2^{k}(-U_2)^{k})$ and $\mathsf L(U_1(-U_4))+\mathsf L(U_2^{k}(-U_2)^{k})=[2(k+1), 5(k+1)-2]$, we obtain that
\begin{equation}
\label{eq5} \mathsf L(U_1(-U_4)U_2^k(-U_2)^k)=[2(k+1), 5(k+1)-2]\,.\\
\end{equation}

Since
\[ [2k+1, 5k+2]\supset \mathsf L(U_1U_3U_4U_2^{k-1}(-U_2)^{k-1})\supset \mathsf L(U_1U_3U_4)+\mathsf L(U_2^{k-1}(-U_2)^{k-1})\]
and $\mathsf L(U_1U_3U_4)+\mathsf L(U_2^{k-1}(-U_2)^{k-1})=[2k+1, 5k+2]$, we obtain that \begin{equation}
\label{eq6} \mathsf L(U_1U_3U_4U_2^{k-1}(-U_2)^{k-1})=[2k+1, 5k+2]\,.\\
\end{equation}

 Since
\[
[2k+1, 4k+2]\supset \mathsf L(U_5^2(-g)^4g^{4k-4}(-g)^{4k-4})\supset \mathsf L(U_5^2(-g)^4)+\mathsf L(g^{4k-4}(-g)^{4k-4}),
\]
$\mathsf L(U_5^2(-g)^4)+\mathsf L(g^{4k-4}(-g)^{4k-4})=[2k+1, 4k]\cup\{4k+2\}$
 and \[4k+1\notin \mathsf L(U_5^2(-g)^4g^{4k-4}(-g)^{4k-4})\] by Lemma \ref{4.2}.3, we obtain that \begin{equation}
\label{eq7} \mathsf L(U_5^2(-g)^{4}g^{4k-4}(-g)^{4k-4})=[2k+1, 4k]\cup \{4k+2\}\,. \\
\end{equation}

Suppose that $k \ge 2$.
Since
\[
[2k, 4k]\supset \mathsf L(U_5(-U_5)g^{4k-4}(-g)^{4k-4})\supset \mathsf L(U_5(-U_5)g^4(-g)^4)+\mathsf L(g^{4k-8}(-g)^{4k-8})\, ,
\]
$\mathsf L(U_5(-U_5)g^4(-g)^4)+\mathsf L(g^{4k-8}(-g)^{4k-8})=[2k, 4k-2]\cup\{4k\}$,
and \[4k-1\notin \mathsf L(U_5(-U_5)g^{4k-4}(-g)^{4k-4})\] by Lemma \ref{4.2}.3, we obtain that \begin{equation}
\label{eq8} \mathsf L(U_5(-U_5)g^{4k-4}(-g)^{4k-4})=[2k, 4k-2]\cup \{4k\} \,.
\end{equation}

\medskip
\noindent
{\bf Step 1.} We prove that for every $L\in  \mathcal L_2\cup \mathcal L_3\cup \mathcal L_4\cup \mathcal L_5$ there exists an $A\in \mathcal B(G)$ such that $L=\mathsf L(A)$.

We distinguish four cases.

\smallskip
First we suppose that $L\in \mathcal L_2$, and we distinguish several subcases.  If $L=y+[3,6]$ with $y\in\N_0$, then we set $A=0^yU_1(-U_1)(e+2g)^2$ and hence $\mathsf L(A)=y+[3,6]=L$ by  Equation \eqref{eq1}. If $L=[2k+1,5k+2]$ with $k\in \N$, then we set $A=U_1U_3U_4U_2^{k-1}(-U_2)^{k-1}$ and hence $\mathsf L(A)=L$ by Equation  \eqref{eq6}. Now we assume that $L=y+2\lceil\frac{k}{3}\rceil+[0,k]$ with $y\in \N_0$ and $k\in \N\setminus\{3\}$.

If  $k\equiv 0\pmod 3$, then $k\ge 6$ and  by Equation \eqref{eq3} we infer that
\[
\mathsf L\big(0^yU_1(-U_1)U_2^t(-U_2)^t\big)=y+[2t+2, 5t+5]=y+2\lceil\frac{k}{3}\rceil+[0,k]=L\,, \text{ where $k=3t+3$}\,.
\]
If  $k\equiv 1\pmod 3$, then by Equation \eqref{eq5} we infer that
\[
\mathsf L\big(0^{y}U_1(-U_4)U_2^{t}(-U_2)^{t}\big)=y+[2t+2, 5t+3]=y+2\lceil\frac{k}{3}\rceil+[0,k]=L\,, \text{ where $k=3t+1$}\,.
\]
If  $k\equiv 2\pmod 3$, then by Equation \eqref{eq4} we infer that
\[
\mathsf L\big(0^yU_1U_3U_2^t(-U_2)^t\big)=y+[2t+2, 5t+4]=y+2\lceil\frac{k}{3}\rceil+[0,k]=L\,, \text{ where $k=3t+2$}\,.
\]

\smallskip
If $L=y+2k+2 \cdot [0,k]\in \mathcal L_3$ with $y\in \N_0$ and $k\in \N$, then we set $A=0^yg^{4k}(-g)^{4k}$ and hence $\mathsf L(A)=L$.

\smallskip
If $L=y+2t+2+(\{0\}\cup[2,2t+3])\in \mathcal L_4$ with $y,t\in \N_0$, then we set  $A=0^y  U_1 (-U_1) g^{4t} (-g)^{4t}$ and obtain that $\mathsf L(A)=y+2t+2+(\{0\}\cup[2,2t+3])=L$ by Equation \eqref{eq2}.

\smallskip
Finally we suppose that $L=y+k+([0,k-2]\cup \{k\})\in \mathcal L_5$ with $k\ge 3$ and $y\in \N_0$, and we distinguish two subcases. If $k=2t$ with $t\ge 2$, then we set $A=0^y U_5 (-U_5) g^{4t-4} (-g)^{4t-4}$ and hence $\mathsf L(A)=y+k+([0,k-2]\cup \{k\})=L$ by Equation \eqref{eq8}. If $k=2t+1$ with $t\ge 1$, then we set $A=0^y U_5^2 (-g)^4 g^{4t-4} (-g)^{4t-4}$ and hence $\mathsf L(A)=y+k+([0,k-2]\cup \{k\})=L$ by Equation \eqref{eq7}.

\medskip
\noindent
{\bf Step 2.} We prove that for every $A\in \mathcal B(G^{\bullet})$, $\mathsf L(A)\in  \mathcal L_2\cup \mathcal L_3\cup \mathcal L_4\cup \mathcal L_5$.

Let $A \in \mathcal B (G^{\bullet})$. We may suppose that $\Delta (\mathsf L (A)) \ne \emptyset$. By Lemma \ref{4.4}.1 we have to distinguish two cases.

\smallskip
\noindent
CASE 1: \,  $\Delta (\mathsf L (A))=\{1\}$.

Then $\mathsf L(A)$ is an interval, say $\mathsf L(A)=[y,y+k]=y+[0,k]$ with $y\ge 2$ and $k\ge 1$. If $k=3$ and $y=2$, then $\mathsf L(A)=[2,5]$, a contradiction to Lemma \ref{4.4}.1. Thus $k=3$ implies  that $y\ge 3$ and hence $\mathsf L(A)\in \mathcal L_2$. If $k\le 2$, then obviously $\mathsf L(A)\in \mathcal L_2$. Suppose that $k\ge 4$.
If $y=2t$ with $t\ge 2$, then $y+k\le 5t$ and hence $y=2t\ge 2\lceil\frac{k}{3}\rceil$ which implies that $\mathsf L(A)\in \mathcal L_2$. Suppose that $y=2t+1$ with $t\in \N$. If $y+k\le 5t+1$, then $y=2t+1\ge 1+ 2\lceil\frac{k}{3}\rceil$ which implies that $\mathsf L(A)\in \mathcal L_2$. Otherwise $y+k=5t+2$ and hence $\mathsf L(A)=[2t+1,5t+2]\in \mathcal L_2$.

\smallskip
\noindent
CASE 2: \,   $2\in \Delta (\mathsf L (A))\subset [1,2]$.

We freely use the classification of minimal zero-sum sequence given in Lemma \ref{lem_zssC2C4}.
Since $2 \in \Delta (\mathsf L (A))$, there are $k \in \N$ and $U_1,\ldots, U_k, V_1,\ldots, V_{k+2} \in \mathcal A (G)$ with $|U_1|\ge |U_2|\ge \ldots \ge |U_k|$ such that
\[
A=U_1\cdot\ldots\cdot U_k=V_1\cdot\ldots\cdot V_{k+2} \quad \text{and} \quad k+1\not\in \mathsf L(A) \,,
\]
and we may suppose  that $k$ is minimal with this property. Then $[\min \mathsf L(A), k]\in \mathsf L(A)$  and there exists $k_0\in [2,k]$ such that $|U_i|\ge 3$ for every $i\in [1,{k_0}]$ and $|U_i|=2$ for every $i\in [k_0+1,k]$. We continue with two simple assertions.

\smallskip
\begin{enumerate}
\item[{\bf A1.}\,] For each two distinct $i,j\in [1,k_0]$, we have that $3\not\in \mathsf L(U_iU_j)$.

\item[{\bf A2.}\,] $|\mathsf L(U_1\cdot\ldots\cdot U_{k_0})|\ge 2$.
\end{enumerate}

\noindent
{\it Proof of} \,{\bf A1}.\, Assume to the contrary that there exist distinct $i,j\in [1,k_0]$ such that $3\in \mathsf L(U_iU_j)$.  This implies that  $k+1\in \mathsf L(A)$, a contradiction. \qed[Proof of {\bf A1}]

\smallskip
\noindent
{\it Proof of} \,{\bf A2}.\, Assume to the contrary that  $|\mathsf L(U_1\cdot\ldots\cdot U_{k_0})|=1$. Then   Lemma \ref{4.2}.2 implies that $\max \mathsf L(A)=\max \mathsf L(U_1\cdot\ldots\cdot U_{k_0})+k-k_0=k$, a contradiction.
\qed[Proof of {\bf A2}]

\medskip
We  use {\bf A1} and {\bf A2} without further mention and freely use  Lemma \ref{lem_zssC2C4} together with  all its notation. We distinguish six subcases.

\medskip
\noindent
CASE 2.1: \, $U_1\in S_5$.

Without loss of generality, we may assume that $U_1=eg^3(e+g)$.  We choose $j\in [2,k_0]$ and start with some preliminary observations.
If $|U_j|=5$, then the fact that $3\not\in \mathsf L(U_1U_j)$ implies  that $U_j=-U_1$.
If $|U_j|=4$, then $3\not\in \mathsf L(U_1U_j)$ implies  that  $U_j\in \{g^2(e+g)^2, g^4, (-g)^4,  (e+g)^4\}$.
If $|U_j|=3$, then $3\not\in \mathsf L(U_1U_j)$ implies  that $U_j\in \{(e+2g)g(e+g), g^2(2g),  (e+g)^2(2g)\}$.

Now we distinguish three cases.

Suppose that  $|U_2|=5$.
Then $U_2=-U_1$ and by symmetry we obtain that $U_j\in \{g^4,(-g)^4\}$ for every $j\in [3,k_0]$. Let $i\in [k_0+1, k]$. If $U_i\neq e^2$, then $4\in U_1U_2U_i$ and hence $k+1\in \mathsf L(A)$, a contradiction. Therefore we obtain that
\[
A=U_1(-U_1)(g^4)^{k_1}((-g)^4)^{k_2}(e^2)^{k_3} \quad \text{ where} \quad  k_1,k_2,k_3\in \N_0\,,
\]
and without loss of generality we may assume that $k_1\ge k_2$. Then it follows that $\mathsf L(A)$ is equal to
\[
k_1-k_2+k_3+\mathsf L(U_1(-U_1)(g^4)^{k_2}((-g)^4)^{k_2})=k_3+k_1-k_2+2k_2+2+(\{0\}\cup[2,2k_2+3]),
\]
which is an element of $\mathcal L_4$.

Suppose that $|U_2|=4$ and there exists $j\in [2,k_0]$ such that $U_j=(-g)^4$, say $j=2$. Let $i\in [3,k_0]$. If $U_i\in \{g^2(e+g)^2, g^2(2g)\}$, then $3\in \mathsf L(U_2U_i)$ and hence $k+1\in \mathsf L(A)$, a contradiction. If $U_i\in \{(e+g)^4, (e+g)^2(2g), (e+2g)g(e+g)\}$, then $4\in \mathsf L(U_1U_2U_i)$ and hence $k+1\in \mathsf L(A)$, a contradiction.
Therefore $U_i\in \{g^4,(-g)^4\}$.
Let $\tau \in [k_0+1,k]$. If $U_{\tau}\in \{(e+2g)^2, (2g)^2, (e+g)(e-g)\}$, then $4\in \mathsf L(U_1U_2U_{\tau})$ and hence $k+1\in \mathsf L(A)$, a contradiction. Therefore $U_{\tau}\in \{e^2, g(-g)\}$.
Therefore we obtain that
\[
A=U_1(g^4)^{k_1}((-g)^4)^{k_2}(g(-g))^{k_3}(e^2)^{k_4} \quad \text{ where} \quad k_1,k_3,k_4\in \N_0 \text{ and }k_2\in \N
\]
and hence $\mathsf L(A)$ is equal to
\[\mathsf L((g^4)^{k_1+1}((-g)^4)^{k_2}(g(-g))^{k_3}(e^2)^{k_4})=k_4+\mathsf L(g^{4k_1+4+k_3}(-g)^{4k_2+k_3})\,
\]
which is in $\mathcal L_3$

Suppose that $|U_2|\le 4$ and for every $j\in [2,k_0]$, we have $U_j\neq (-g)^4$. Then $U_j\in \{g^2(e+g)^2, g^4,   (e+g)^4, (e+2g)g(e+g), g^2(2g),  (e+g)^2(2g) \}$.  Since $\supp(U_1\cdot\ldots\cdot U_{k_0})\subset\{e,g,2g, e+g,e+2g\}$ and $\mathsf v_e(U_1\cdot\ldots\cdot U_{k_0})=1$,   Lemma \ref{4.4}.3 implies that  $|\mathsf L(U_1\cdot \ldots\cdot U_{k_0})|=1$, a contradiction.

\medskip
\noindent
CASE  2.2: \, $U_1\in S_4^4$.

Without loss of generality, we may assume that $U_1=eg(2g)(e+g)$. Let $j\in [2,k_0]$.

Suppose that $|U_j|=4$. Since  $3\not\in \mathsf L(U_1U_j)$, we obtain that  $U_j\in \{g^2(e+g)^2, g^4,(e+g)^4\}$. Thus  $U_1U_j=W_1W_2$ with $|W_1|=5$, where $W_1,W_2$ are atoms and hence we are back to CASE 2.1.

Suppose that $|U_j|=3$.  Since $3\not\in \mathsf L(U_1U_j)$, we obtain that   $U_j\in \{(e+2g)g(e+g), g^2(2g), (e+g)^2(2g)\}$. If $U_j\in \{g^2(2g), (e+g)^2(2g)\}$, then $U_1U_j=W_1W_2$ with $|W_1|=5$, where $W_1,W_2$ are atoms and hence we are back to CASE 2.1.  Thus it remains to consider the case where $U_j=(e+2g)g(e+g)$.

Therefore we have
\[
U_1\cdot\ldots\cdot U_{k_0}=U_1\big((e+2g)g(e+g)\big)^{k_1} \quad \text{ where} \quad  k_1\in\N \,.
\]
Since $\supp(U_1\cdot\ldots\cdot U_{k_0})\subset\{e,g,2g, e+g,e+2g\}$ and $\mathsf v_e(U_1\cdot\ldots\cdot U_{k_0})=1$,   Lemma \ref{4.4}.3 implies that  $|\mathsf L(U_1\cdot \ldots\cdot U_{k_0})|=1$, a contradiction.

\medskip
\noindent
CASE 2.3: \, $U_1\in S_4^3$ and  for every $i\in [2,k_0]$, we have  $U_i\not\in S_4^4$.

Without loss of generality, we may assume that $U_1=eg^2(e+2g)$. Let $j\in [2,k_0]$.

Suppose that $|U_j|=4$. Since $3\not\in \mathsf L(U_1U_j)$, we obtain that  $U_j\in  \{-U_1, g^2(e+g)^2, g^2(e-g)^2,(e+g)^4, (e-g)^4, g^4\}$.    If $U_j\in \{g^2(e+g)^2, g^2(e-g)^2,(e+g)^4, (e-g)^4\}$, then $U_1U_j=W_1W_2$ with $|W_1|=5$, where $W_1,W_2$ are atoms and hence we are back to  CASE 2.1. Thus it remains to consider the cases where $U_j=-U_1$ or $U_j=g^4$.

Suppose that $|U_j|=3$.  Since $3\not\in \mathsf L(U_1U_j)$, we obtain that   $U_j\in \{eg(e-g), (e+2g)g(e+g), g^2(2g), (e+g)^2(2g), (e-g)^2(2g)\}$. If $U_j\in \{eg(e-g), (e+2g)g(e+g)\}$, then $U_1U_j=W_1W_2$ with $|W_1|=5$, where $W_1,W_2$ are atoms and hence we are back to  CASE 2.1. If $U_j\in \{(e+g)^2(2g), (e-g)^2(2g)\}$, then  $U_1U_j=W_1W_2$  with $W_1\in S_4^4$, where $W_1,W_2$ are atoms and hence we are back to CASE 2.2. Thus it remains to consider the case where $U_j=g^2(2g)$.

If $U_i\neq -U_1$ for every $i\in [2,k_0]$, then $U_1\cdot\ldots\cdot U_{k_0}=U_1(g^4)^{k_1}(g^2(2g))^{k_2}$ where $k_1,k_2\in \N_0$. Since $\supp(U_1\cdot\ldots\cdot U_{k_0})\subset\{e,g,2g, e+g,e+2g\}$ and $\mathsf v_e(U_1\cdot\ldots\cdot U_{k_0})=1$,   Lemma \ref{4.4}.3 implies that  $|\mathsf L(U_1\cdot \ldots\cdot U_{k_0})|=1$, a contradiction.
Thus there exists some $i\in [2,k_0]$, say $i=2$, such that $U_2=-U_1$. By symmetry we obtain that $k_0=2$. Let $\tau\in [3,k]$. If $U_{\tau}\in \{(2g)^2, (e+g)(e-g)\}$, then $4\in \mathsf L(U_1U_2U_{\tau})$ and hence $k+1\in \mathsf L(A)$, a contradiction. Therefore
 $A=U_1(-U_1)(e^2)^{k_1}((e+2g)^2)^{k_2}(g(-g))^{k_3}$ where
 $k_1,k_2,k_3\in\N_0$. Since $[\min \mathsf L(A), 2+k_1+k_2+k_3]\subset \mathsf L(A)$, we obtain that  $\mathsf L(A)=[\min \mathsf L(A), 2+y]\cup\{4+y\}$ where $y=k_1+k_2+k_3\in\N_0$. For every atom $V$ dividing $A$, we have that $|V|=2$ or $|V|=4$.  Thus $\min \mathsf L(A)\ge 2+\frac{y}{2}$ which implies that $\mathsf L(A)\in \mathcal L_5$.

\medskip
\noindent
CASE 2.4: \, $U_1\in S_4^2$ and for every $i\in [2,k_0]$, we have   $U_i\not\in S_4^4\cup S_4^3$.

Without loss of generality, we may assume that $U_1=g^2(e+g)^2$. Let $j\in [2,k_0]$.

Suppose that $|U_j|=4$. If $U_j\in  \{g^2(e-g)^2, (-g)^2(e+g)^2, (-g)^4, (e-g)^4\}$,  then $3\in \mathsf L(U_1U_j)$, a contradiction.  Thus $U_j\in \{U_1, -U_1, g^4, (e+g)^4\}$.

Suppose that $|U_j|=3$.   If $U_j\in \{(e+2g)(-g)(e-g), (-g)^2(2g), (e-g)^2(2g)\}$, then $3\in \mathsf L(U_1U_j)$, a contradiction.  If $U_j\in \{eg(e-g), e(-g)(e+g)\}$, then $U_1U_j=W_1W_2$ with $|W_1|=5$, where $W_1,W_2$ are atoms and hence we are back to  CASE 2.1. If $U_j=e(2g)(e+2g)$, then  $U_1U_j=\big(e(e+g)g(2g)\big) \big(g(e+g)(e+2g)\big)$ and $e(e+g)g(e+2g)\in S_4^4$,  going back to  CASE 2.2.  Thus it remains to consider the case where $U_j=g^2(2g)$ or $U_j= (e+g)^2(2g)$.

If $U_i\neq -U_1$ for every $i\in [2,k_0]$, then $\supp(U_1\cdot\ldots\cdot U_{k_0})\subset \{g,2g, e+g,e+2g\}$ is half-factorial by Lemma \ref{4.4}.3, a contradiction.
Thus there exists some $i\in [2,k_0]$, say $i=2$, such that $U_2=-U_1$. By symmetry we obtain that $\{U_1, \ldots, U_{k_0}\}=\{U_1,-U_1\}$. Let $\tau\in [k_0+1, k]$. If $U_{\tau}\in \{e^2, (2g)^2, (e+2g)^2\}$, then $4\in \mathsf L(U_1U_2U_{\tau})$ and $k+1\in \mathsf L(U_1U_2U_{\tau})$, a contradiction. Therefore $A=U_1^{k_1}(-U_1)^{k_2}(g(-g))^{k_3}((e+g)(e-g))^{k_4}$ where $k_1,k_2\in \N$ and $k_3,k_4\in\N_0$.  If $k_1+k_2\ge 3$, by symmetry we assume that $k_1\ge 2$, then $U_1^2(-U_1)=g^4  (-g)^2(e+g)^2  (e+g)(e-g)  (e+g)(e-g)$ and hence $4\in \mathsf L(U_1^2(-U_1))$ which implies that $k+1\in \mathsf L(A)$, a contradiction. Thus $k_1=k_2=1$ and hence $A=U_1(-U_1)(g(-g))^{k_3}((e+g)(e-g))^{k_4}$ where  $k_3,k_4\in\N_0$. Since $[\min \mathsf L(A), 2+k_3+k_4]\in \mathsf L(A)$, we obtain that $\mathsf L(A)=[\min \mathsf L(A), 2+y]\cup\{4+y\}$ where $y=k_3+k_4\in\N_0$. For every atom $V$ dividing $A$, we have that $|V|=2$ or $|V|=4$.  Thus $\min \mathsf L(A)\ge 2+\frac{y}{2}$ which implies that $\mathsf L(A)\in \mathcal L_5$.

\medskip
\noindent
CASE 2.5: \, $U_1\in S_4^1$ and  for every $i\in [2,k_0]$, we have   $U_i\not\in S_4^4\cup S_4^3\cup S_4^2$.

Without loss of generality, we may assume that $U_1=g^4$. Let $j\in [2,k_0]$.

Suppose that $|U_j|=4$. If $U_j\in  \{(e+g)^4, (e-g)^4\}$,  then $U_1U_j=W_1W_2$ with $W_1\in S_4^2$, where $W_1,W_2$ are atoms and hence we are back to  CASE 2.4. Thus it remains to consider the case where $U_j=U_1$ or $U_j=-U_1$.

Suppose that $|U_j|=3$. If $U_j\in \{(-g)^2(2g)\}$, then $3\in \mathsf L(U_1U_j)$, a contradiction.  If $U_j\in \{e(-g)(e+g), (e+2g)(-g)(e-g)\}$, then $U_1U_j=W_1W_2$ with $|W_1|=5$, where $W_1,W_2$ are atoms and hence we are back to  CASE 2.1. If $U_j\in\{(e+g)^2(2g), (e-g)^2(2g)\}$, then  $U_1U_j=W_1W_2$ with $W_1\in S_4^2$, where $W_1,W_2$ are atoms and hence we are back to CASE 2.4. If $U_j=e(2g)(e+2g)$, then $U_1U_j=W_1W_2$ with $W_1\in S_4^3$, where $W_1,W_2$ are atoms and hence we are back to  CASE 2.3. Thus it remains to consider the case where $U_j=g^2(2g)$, or $U_j=eg(e-g)$, or $U_j=(e+2g)g(e+g)$.

First, suppose that $U_i\neq -U_1$ for every $i\in [2,k_0]$. Then
\[
U_1\cdot\ldots\cdot U_{k_0}=U_1^{k_1}(eg(e-g))^{k_2}((e+2g)g(e+g))^{k_3}(g^2(2g))^{k_4}\,,
\]
where $ k_1\in\N $ and  $k_2,k_3, k_4\in \N_0$.
If  $k_2\ge 1$ and  $k_3\ge 1$, then
$eg(e-g)  (e+2g)g(e+g)=eg^2(e+2g)  (e+g)(e-g)$, $eg^2(e+2g)\in S_4^3$
and hence we are back to  CASE 2.3. Thus we may assume that $k_2=0$ or $k_3=0$. Since $\{g,2g, e+g, e+2g\}$ and $\{g, 2g, e, e-g\}$ are both half-factorial by Lemma \ref{4.4}.3, we obtain that $|\mathsf L(U_1\cdot\ldots\cdot U_{k_0})|=1$, a contradiction.

Second, suppose that there exists some $i\in [2,k_0]$, say $i=2$, such that $U_2=-U_1$. By symmetry we obtain that $\{U_1,\ldots, U_{k_0}\}=\{U_1,-U_1\}$. Since $4\in \mathsf L(U_1\cdot U_2\cdot (2g)^2)$, $5\in \mathsf L(U_1 U_2 e^2  (e-g)(e+g))$, and $5\in \mathsf L(U_1 U_2 (e+2g)^2 (e-g)(e+g))$, we obtain that
\[
\{U_{k_0+1}, \ldots,U_k\}\subset \{(e+g)(e-g), g(-g)\} \quad \text{ or} \]
\[  \{U_{k_0+1}, \ldots,U_k\}\subset \{e^2, (e+2g)^2, g(-g)\} \,.
\]
This implies that
\[
A=(g^4)^{k_1}((-g)^4)^{k_2}((e+g)(e-g))^{k_3}(g(-g))^{k_4}  \quad \text{ or}\]
 \[ A=(g^4)^{k_1}((-g)^4)^{k_2}(e^2)^{k_3}((e+2g)^2)^{k_4}(g(-g))^{k_5} \,,
\]
where $k_1,k_2\in \N$ and $k_3,k_4,k_5\in \N_0$.

Suppose that $A=(g^4)^{k_1}((-g)^4)^{k_2}((e+g)(e-g))^{k_3}(g(-g))^{k_4}$,  where $k_1,k_2\in \N$ and $k_3,k_4,k_5\in \N_0$.
If $k_1\ge 2$ and $k_3\ge 2$, then $g^4\  g^4  (-g)^4 (e+g)(e-g) (e+g)(e-g)=\big(g(-g)\big)^4  g^2(e+g)^2 g^2(e-g)^2$ and hence $6\in \mathsf L(g^4 g^4 (-g)^4 (e+g)(e-g) (e+g)(e-g))$. Thus $k+1\in \mathsf L(A)$, a contradiction. Therefore by symmetry  $k_3=1$ or $k_1=k_2=1$. If $k_3=1$, then $\mathsf L(A)=1+\mathsf L((g^4)^{k_1}((-g)^4)^{k_2}(g(-g))^{k_4})\in \mathcal L_3$. If $k_1=k_2=1$, then $\mathsf L(A)=[\min \mathsf L(A), 2+y]\cup\{4+y\}$ where $y=k_3+k_4\in\N_0$. For every atom $V$ dividing $A$, we have that $|V|=2$ or $|V|=4$.  Thus $\min \mathsf L(A)\ge 2+\frac{y}{2}$ which implies that $\mathsf L(A)\in \mathcal L_5$.

Suppose that $A=(g^4)^{k_1}((-g)^4)^{k_2}(e^2)^{k_3}((e+2g)^2)^{k_4}(g(-g))^{k_5}$, where $k_1,k_2\in \N$ and $k_3,k_4,k_5\in \N_0$. If $k_1\ge 2$, $k_3\ge 1$, and $k_4\ge 1$, then $g^4 g^4 (-g)^4 e^2 (e+2g)^2=\big(g(-g)\big)^4 \big(e(e+2g)g^2\big)^2$ and hence $6\in \mathsf L(g^4 g^4 (-g)^4 e^2 (e+2g)^2)$. Thus $k+1\in \mathsf L(A)$, a contradiction. Therefore by symmetry  $k_3=0$, or $k_4=0$, or $k_1=k_2=1$. If $k_3=0$ or $k_4=0$, then $\mathsf L(A)=k_3+k_4+\mathsf L((g^4)^{k_1}((-g)^4)^{k_2}(g(-g))^{k_5})\in \mathcal L_3$. If $k_1=k_2=1$, then $\mathsf L(A)=[\min \mathsf L(A), 2+y]\cup\{4+y\}$ where $y=k_3+k_4+k_5\in\N_0$. For every atom $V$ dividing $A$, we have that $|V|=2$ or $4$.  Thus $\min \mathsf L(A)\ge 2+\frac{y}{2}$ which implies that $\mathsf L(A)\in \mathcal L_5$.

\medskip
\noindent
CASE 2.6: \, $|U_1|=3$.

Let $j\in [2,k_0]$. We distinguish three subcases.

First, we suppose that $U_1\in S_3^3$,  and without restriction we may  assume that $U_1=eg(e-g)$. If $U_j=-U_1$, then $3\in\mathsf L(U_1U_j)$, a contradiction. If $U_j\in \{(-g)^2(2g), (e+g)^2(2g), e(2g)(e+2g)\}$, then $U_1U_j=W_1W_2$ with $W_1 \in S_4^4$ where $W_1,W_2$ are atoms and hence we are back to  CASE 2.2.
If $U_j\in \{(e+2g)g(e+g), (e+2g)(-g)(e-g)\}$, then $U_1U_j=W_1W_2$ with $W_1 \in S_4^3$ where $W_1,W_2$ are atoms and hence we are back to  CASE 2.3.  If $U_j=U_1$, then $U_1U_j=W_1W_2$ with $W_1 \in S_4^2$ where $W_1,W_2$ are atoms and hence we are back to  CASE 2.4. Thus it remains to consider the case where $U_j=g^2(2g)$ or $(e-g)^2(2g)$. Then  $U_1\cdot\ldots\cdot U_{k_0}=U_1(g^2(2g))^{k_1}((e-g)^2(2g))^{k_2}$ where $k_1,k_2\in \N_0$. Since $\{e,g,2g,e-g\}$ is half-factorial by Lemma \ref{4.4}.3, we obtain that  $|\mathsf L(U_1\cdot\ldots\cdot U_{k_0})|=1$, a contradiction.

Second, we suppose  that $U_1\in S_3^2$, and without restriction we may assume that $U_1=g^2(2g)$ and $U_j\not\in S_3^3$. If $U_j=-U_1$, then $3\in \mathsf L(U_1U_j)$. If $U_j=U_1$, then $U_1U_j=W_1W_2$ with $W_1 \in S_4^1$ where $W_1,W_2$ are atoms and hence we are back to  CASE 2.5. If $U_j\in \{(e+g)^2(2g), (e-g)^2(2g)\}$, then $U_1U_j=W_1W_2$ with $W_1 \in S_4^2$ where $W_1,W_2$ are atoms and hence we are back to  CASE 2.4. If $U_j=e(2g)(e+2g)$, then $U_1U_j=W_1W_2$ with $W_1 \in S_4^3$ where $W_1,W_2$ are atoms and hence we are back to  CASE 2.3.

Third, we suppose  that $U_1\in S_3^1$,  and without restriction we assume that  $U_j\in S_3^1$. Thus $3\in \mathsf L(U_1U_j)$, a contradiction.
 \end{proof}

\subsection{The system of sets of lengths of $C_2^4$}

Now we  give a complete description of the system of sets of lengths of $C_2^4$.

\medskip
\begin{theorem} \label{4.8}
$\quad \mathcal L (C_2^4) = \mathcal L_1\cup \mathcal L_2\cup \mathcal L_3\cup \mathcal L_4\cup \mathcal L_5\cup \mathcal L_6\cup \mathcal L_7\cup \mathcal L_8$, where
\begin{align*}
\mathcal L_1&=\big\{ \{y\} \mid y \in \N_0\big\}, \\
 \mathcal L_2&= \big\{ y+2k+3 \cdot [0,k] \mid y, k \in \N_0 \big\}, \\
 \mathcal L_3&= \big\{y+[2,3],  y+[2,4], y+[3,6], y+[3,7],y+[4,9] \mid y \in \N_0\big \}\cup\\
   & \quad \ \big\{ y+[m,m+k] \mid y \in \N_0, k \ge 6, m \ \text{\rm minimal with} \ m+k\le 5m/2 \big\} \\
   &=\big\{y+\left\lceil\frac{2k}{3}\right\rceil+[0,k]\mid y\in \N_0, k \in \N\setminus\{1,3\}\big\}\cup  \\
   & \quad \  \{y+3 + [0,3], y+2+[0,1]\mid y\in \N_0\}, \\
 \mathcal L_4&=\bigl\{ y + 2k
      + 2 \cdot [0, k] \, \bigm|\, y ,\, k \in \N_0 \bigr\}, \\
      \mathcal L_5&= \{ y+k+2+([0,k]\cup\{k+2\})\mid  y\in\N_0, k\in \N \}, \\
      \mathcal L_6&=\{ y+2\left\lceil \frac{k}{3}\right \rceil +2+(\{0\}\cup [2,k+2])\mid  y\in\N_0, k\ge 5 \text{ or }k=3\}, \\
      \mathcal L_7&=\{y+2k+3+\{0,1,3\}+3 \cdot [0,k]\mid y, k\in \N_0\}\ \cup\\
      & \quad \ \{y+2k+4+\{0,1,3\}+3 \cdot [0,k]\cup \{y+5k+8\}\mid y, k\in \N_0\}, \text{ and } \\
 \mathcal L_8&=\{y+2k+3+\{0,2,3\}+3 \cdot [0,k]\mid y, k\in \N_0\}\ \cup\\
           & \quad \ \{y+2k+4+\{0,2,3\}+3 \cdot [0,k]\cup \{y+5k+9\}\mid y, k\in \N_0\} . \\  \end{align*}
\end{theorem}

We note that the system of sets of lengths of $C_2^4$ is richer than that of the other groups we considered. A reason for this is that the set $\Delta^{\ast}(C_2^4)$ is largest, namely $\{1,2,3\}$ (this fact was also crucial in the proof of Theorem \ref{3.5}).
We recall some useful facts in the lemma below.

\medskip
\begin{lemma} \label{4.5}
Let $G=C_2^4$, and let $A \in \mathcal B (G)$.
\begin{enumerate}
\item $\Delta (G) = [1,3]$, and if \ $3 \in \Delta ( \mathsf L (A))$, then $\Delta ( \mathsf L (A)) = \{3\}$ and there is a basis $(e_1, \ldots, e_4)$ of $G$ such that $\supp (A) \setminus \{0\} = \{e_1, \ldots, e_4, e_1+ \ldots +e_4\}$.

\item $\rho_{2k+1} (G) = 5k+2$ for all $k \in \N$.
\end{enumerate}
\end{lemma}

\begin{proof}
1. The first statement follows from \cite[Theorem 6.8.3]{Ge-HK06a}, and the second statement from \cite[Lemma 3.10]{Ge-Sc16b}.

\smallskip
2. See \cite[Theorem 6.3.4]{Ge-HK06a}.
 \end{proof}

\medskip

In the following result we characterize which intervals are sets of lengths for $C_2^4$. It turns out  that, with a single exception, the sole restriction is the one implied by elasticity.

\begin{proposition}
\label{prop_C24_int}
Let $G=C_2^4$ and let $2 \le l_1 \le l_2$ be integers. Then $[l_1,l_2] \in \mathcal L (G)$ if and only if $l_2/l_1 \le 5/2$ and $(l_1, l_2 )  \neq (2,5)$.
\end{proposition}

\begin{proof}
Suppose that $[l_1,l_2] \in \mathcal L (G)$ with integers $2 \le l_1 \le l_2$. Then \eqref{inequality-elasticity} implies that  $l_2/l_1 \le \rho(G) = 5/2 $.  Moreover, $[2,5]=[2, \mathsf D (G)]  \notin \mathcal L (G)$ by \cite[Theorem 6.6.3]{Ge-HK06a}.

Conversely, we need to show that  for integers $2 \le l_1 \le l_2$ with $(l_1, l_2) \neq (2,5)$  and $l_2/l_1 \le 5/2$, we have   $[l_1,l_2] \in \mathcal L (G)$.
We start with an observation that reduces the problem to constructing these sets of intervals for extremal choices of the endpoints.

Let $k \in \N$. If $m \in \N$ such that $[m,m+k] \in \mathcal L (G)$, then $y + [m,m+k] \in \mathcal L (G)$ for all $y \in \N_0$. Thus let $m_k=\max\{2,\lceil\frac{2k}{3}\rceil \}$ if $k\in \N\setminus\{3\}$ and $m_3=3$.  Therefore we only need  to prove that  $[m_k,m_k+k] \in \mathcal L (G)$.

For  $k\in [1,5]$ we are going to realize sets $[m_k,m_k+k]$ as sets of lengths. Then we handle the case $k \ge 6$.

If $k \in \{1,3\}$, then the sets $[2,3],  [3,6] \in \mathcal L (C_2^3) \subset \mathcal L (G)$. To handle the case $k=2$, we have to show that $[2,4] \in \mathcal L (G)$. Let $(e_1,\ldots, e_4)$ be a basis of $G$ and $e_0=e_1+\ldots+e_4$. If
\[
U_1= e_0 \cdot \ldots \cdot e_4 \quad \text{and} \quad  U_2 = e_1e_2(e_1+e_3)(e_2+e_4)(e_3+e_4),
\]
then $\max \mathsf L (U_1U_2) < 5$, and
\[
\begin{aligned}
U_1 U_2 & = \Big(e_0e_1e_2(e_3+e_4)\Big) \Big((e_1+e_3)e_1e_3\Big) \Big((e_2+e_4)e_2e_4\Big)   \\
 & = \Big( e_0(e_1+e_3)(e_2+e_4)\Big) \Big( e_1^2\Big) \Big(e_2^2\Big) \Big( (e_3+e_4)e_3e_4\Big) \,,
\end{aligned}
\]
shows that $\mathsf L (U_1U_2) = [2,4]$. It remains to verify  the following assertions.

\begin{enumerate}
\item[{\bf A1.}\,]  $[3,7] \in \mathcal L (G)$ (this settles the case $k=4$).

\item[{\bf A2.}\,]  $[4,9] \in \mathcal L (G)$ (this settles the case $k=5$).

\item[{\bf A3.}\,]  Let $k \ge 6$. Then $[\lceil\frac{2k}{3}\rceil,\ \lceil\frac{2k}{3}\rceil+k] \in \mathcal L (G)$.

\end{enumerate}

\smallskip
{\it Proof of \,{\bf A1}}.\, Clearly,
\[
\begin{aligned}
U_1 & = e_0 \cdot \ldots \cdot e_4, \ U_2 = e_1e_2(e_1+e_3)(e_2+e_4)(e_3+e_4), \quad \text{and} \\
U_3 & = (e_1+e_3)(e_2+e_4)e_3e_4(e_1+e_2)
\end{aligned}
\]
are minimal zero-sum sequences of lengths $5$. Since
\[
\begin{aligned}
U_1U_2U_3 & = \Big( e_0 (e_1+e_2)(e_3+e_4) \Big) \Big( e_1^2 \Big) \Big( e_2^2 \Big) \Big( e_3^2 \Big)\Big( e_4^2 \Big) \Big( (e_1+e_3)^2\Big) \Big( (e_2+e_4)^2 \Big) \\
 & = \Big( e_0(e_1+e_2)(e_3+e_4) \Big) \Big( (e_1+e_3)e_1e_3 \Big)^2 \Big( (e_2+e_4)^2 \Big) \Big( e_2^2 \Big) \Big( e_4^2 \Big) \\
 & = \Big( e_0 (e_1+e_2)(e_3+e_4) \Big) \Big( (e_1+e_3)e_1e_3 \Big)^2 \Big( (e_2+e_4)e_2e_4 \Big)^2 \\
 & = U_2 \Big( e_0(e_1+e_2)(e_1+e_3)e_1e_4 \Big) \Big( (e_2+e_4)e_2e_4 \Big) \Big( e_3^2 \Big) \,,
\end{aligned}
\]
it follows that $\mathsf L (U_1U_2U_3) = [3,7]$.

\smallskip
{\it Proof of \,{\bf A2}}.\, We use the same notation as in {\bf A1}, set $U_4 = (e_1+e_2)(e_1+e_3)(e_2+e_4)(e_3+e_4)$, and assert that
$\mathsf L (U_1^2 U_2 U_4) = [4,9]$. Clearly, $4 \in \mathsf L (U_1^2 U_2 U_4)$ and $\max \mathsf L (U_1^2 U_2 U_4) < 10$. Since
\[
\begin{aligned}
U_1^2U_2U_4 & = \Big(e_0e_1e_2(e_3+e_4)\Big) \Big((e_1+e_3)e_1e_3\Big) \Big((e_2+e_4)e_2e_4\Big) U_1U_4 \\
            & = \Big(e_0(e_1+e_3)(e_2+e_4)\Big) \Big( e_1^2 \Big) \Big( e_2^2 \Big) \Big((e_3+e_4)e_3e_4\Big) U_1U_4   \\
            & = \prod_{\nu=0}^4 \big( e_{\nu}^2 \big) U_2 U_4 \\
            & = \Big( (e_1+e_3)^2\Big) \Big( (e_2+e_4)^2\Big) \Big( (e_3+e_4)e_3e_4 \Big)^2 \Big(e_0^2\Big) \Big( e_1^2\Big) \Big(e_2^2\Big) \Big( (e_1+e_2)e_1e_2 \Big) \\
            & = \Big( (e_1+e_3)^2\Big) \Big( (e_2+e_4)^2\Big) \Big( (e_3+e_4)^2 \Big) \Big( e_3^2 \Big) \Big( e_4^2 \Big) \Big(e_0^2\Big) \Big( e_1^2\Big) \Big(e_2^2\Big) \Big( (e_1+e_2)e_1e_2 \Big)\,,
\end{aligned}
\]
the assertion follows.

\smallskip
{\it Proof of \,{\bf A3}}.\, We proceed by induction on $k$. For $k=6$, we have to verify that $[4,10] \in \mathcal L (G)$. We use the same notation as in {\bf A1}, and assert that
$\mathsf L (U_1^2 U_2^2) = [4,10]$. Clearly, $\{4, 10\} \subset \mathsf L (U_1^2 U_2^2) \subset [4,10]$. Since
\[
\begin{aligned}
U_1^2 U_2^2 & = \Big(e_0e_1e_2(e_3+e_4)\Big) \Big((e_1+e_3)e_1e_3\Big) \Big((e_2+e_4)e_2e_4\Big) U_1U_2   \\
 & = \Big(e_0e_1e_2(e_3+e_4)\Big)^2 \Big((e_1+e_3)e_1e_3\Big)^2 \Big((e_2+e_4)e_2e_4\Big)^2 \\
 & = \prod_{\nu=0}^4 \big( e_{\nu}^2 \big) U_2^2 \\
 & = \Big( e_0(e_1+e_3)(e_2+e_4)\Big)^2 \Big( e_1^2\Big)^2 \Big(e_2^2\Big)^2 \Big( (e_3+e_4)e_3e_4\Big)^2 \\
 & = \Big( (e_1+e_3)^2\Big) \Big( (e_2+e_4)^2\Big) \Big( (e_3+e_4)e_3e_4 \Big)^2 \Big(e_0^2\Big) \Big( e_1^2\Big)^2 \Big(e_2^2\Big)^2
\end{aligned}
\]
it follows that $[5,9] \subset \mathsf L (U_1^2 U_2^2)$, and hence $\mathsf L (U_1^2 U_2^2) = [4,10]$.

If $k=7$, then $[5,12]\supset \mathsf L(U_1^3U_2U_3)\supset \mathsf L(U_1U_2U_3)+\mathsf L(U_1^2) = [3,7]+\{2,5\}=[5,12]$ which implies that $[5,12] \in \mathcal L (G)$. If $k=8$, then $[6,14]\supset \mathsf L(U_1^4U_2U_4)\supset \mathsf L(U_1^2U_2U_4)+\mathsf L(U_1^2) = [4,9]+\{2,5\}=[6,14]$ which implies that $[6,14] \in \mathcal L (G)$. Suppose that $k \ge 9$, and that the assertion holds for all $k' \in [6, k-1]$. Then the set $[\lceil\frac{2(k-3)}{3}\rceil, \lceil\frac{2(k-3)}{3}\rceil+k-3] \in \mathcal L (G)$. This implies that $[\lceil\frac{2k}{3}\rceil, \lceil\frac{2k}{3}\rceil+k]=[\lceil\frac{2(k-3)}{3}\rceil, \lceil\frac{2(k-3)}{3}\rceil+k-3]+\{2,5\} \in \mathcal L (G)$.
 \end{proof}

We now proceed to prove Theorem \ref{4.8}.

\begin{proof}[Theorem \ref{4.8}]
Let $(e_1, e_2,e_3,e_4)$ be a basis of  $G=C_2^4$. We set $e_0 = e_1+e_2+e_3+e_4$, $U =e_0e_1e_2e_3e_4$, and $V=e_1e_2e_3(e_1+e_2+e_3)$.

\medskip
\noindent
{\bf Step 0.}  Some elementary constructions.

Let $t_1\ge 2$, $t_2\ge 2$, $t=t_1+t_2$, and \begin{align*}
L_{t_1,t_2}=\left\{\begin{aligned}
&\{t\}\cup [t+2, 5\lfloor t_1/2\rfloor+4(t/2-\lfloor t_1/2\rfloor)] &&\text{ if $t$ is even}\,,\\
&\{t\}\cup [t+2, 5\lfloor t_1/2\rfloor+4((t-1)/2-\lfloor t_1/2\rfloor)+1] &&\text{ if $t$ is odd}\,.
\end{aligned}\right.
\end{align*}
Since $\mathsf L(U^2V^2)=\{4\}\cup[6,9]$, we have that
$\mathsf L(U^{t_1}V^{t_2})\supset \mathsf L(U^2V^2)+\mathsf L(U^{t_1-2}V^{t_2-2})=L_{t_1,t_2}$.  Note that for every atom $W$ dividing $U^{t_1}V^{t_2}$, we have
$$W=\left\{\begin{aligned}
&U &\text{ if } |W|=5\,,\\
&V &\text{ if } |W|=4\,,\\
&e_0e_4(e_1+e_2+e_3) &\text{ if } |W|=3\,.
\end{aligned}
\right.$$
Assume to the contrary that $t+1\in \mathsf L(U^{t_1}V^{t_2})$. Then there exist $t_3,t_4,t_5\in \N_0$  and atoms $W_1,\ldots, W_{t_3+t_4+1}$ such that
$U^{t_3}V^{t_4}=W_1\ldots W_{t_3+t_4+1}$ with $t_3+t_4\ge 2$, $t_5\le \min \{t_3,t_4\}$, $|W_i|=3$ for $i\in [1, t_5]$, and $|W_i|=2$ for $i\in [t_5+1, t_3+t_4+1]$.
It follows that $5t_3+4t_4=3t_5+2(t_3+t_4+1-t_5)\le 3t_3+2t_4+2$ and hence $t_3+t_4\le 1$, a contradiction.
Therefore $t+1\not\in \mathsf L(U^{t_1}V^{t_2})$ and  \begin{equation}\label{eq9}\mathsf L(U^{t_1}V^{t_2})=L_{t_1,t_2}.\end{equation}

Note that for every atom $W$ dividing $U^rV$ with $r\ge 2$ and $e_1+e_2+e_3\t W$, we have $W=V$ or $W=e_0e_4(e_1+e_2+e_3)$.
It follows that for all $r \ge 2$
\begin{align}\label{eq10}
 &\mathsf L(U^rV)\nonumber \\
=&\big(1+\mathsf L(U^r)\big)\cup \big(1+\mathsf L(e_1^2e_2^2e_3^2U^{r-1})\big)\\ \nonumber
=&\left\{\begin{aligned}
&r+1+\{0,2,3\}+3 \cdot [0, r/2-1], &&\text{ if $r$ is even}\,,\\
&r+1+\{0,2,3\}+3 \cdot [0,  (r-1)/2-1]\cup \{r+1+(3r-3)/2+2\}, &&\text{ if $r$ is odd}\,.
\end{aligned}\right.
\end{align}

Note that for every atom $W$ dividing $U^rVe_4^2e_0^2$ with $r\ge 2$ and $e_1+e_2+e_3\t W$, we have $W=V$ or $W=e_0e_4(e_1+e_2+e_3)$.
It follows that for all $r \ge 2$
\begin{align}\label{eq11}
&\mathsf L(U^rVe_4^2e_0^2)\nonumber \\
=&\big(1+\mathsf L(U^re_4^2e_0^2)\big)\cup \big(1+\mathsf L(U^{r+1})\big)\\ \nonumber
=&\left\{\begin{aligned}
&r+2+\{0,1,3\}+3 \cdot [0, (r+1)/2-1], &&\text{ if $r$ is odd}\,,\\ \nonumber
&r+2+\{0,1,3\}+3 \cdot [0,  r/2-1]\cup \{r+2+3r/2+1\}, &&\text{ if $r$ is even}\,.
\end{aligned}\right.
\end{align}.

\smallskip
\noindent
{\bf Step 1.} We prove that for every $L\in  \mathcal L_2\cup \mathcal L_3\cup \mathcal L_4\cup \mathcal L_5\cup \mathcal L_6 \cup \mathcal L_7 \cup \mathcal L_8$, there exists an $A\in \mathcal B(G)$ such that $L=\mathsf L(A)$. We  distinguish seven cases.

If $L=y+2k+3\cdot [0,k]\in \mathcal L_2$ with $y, k \in \N_0$, then
$L = \mathsf L ( 0^y U^{2k} ) \in \mathcal L (G)$.

If $L \in \mathcal L_3$, then the claim follows from Proposition \ref{prop_C24_int}.

If $L= y + 2k + 2 \cdot [0, k] \in \mathcal L_4$ with  $y, k \in \N_0$, then Proposition \ref{3.3}.4 implies that $L \in \mathcal L (C_2^3) \subset \mathcal L (G)$.

Suppose that $L=y+k+2+([0,k]\cup \{k+2\})\in \mathcal L_5$ with $k\in \N$ and $y\in \N_0$. Note that $\mathsf L(V^2 (e_1+e_4)^{2}(e_2+e_4)^2(e_3+e_4)^2(e_1+e_2+e_3+e_4)^2)=[4,6]\cup\{8\}$. If $k$ is even, then we set $A=0^y V^2 (e_1+e_4)^{k}(e_2+e_4)^k(e_3+e_4)^k(e_1+e_2+e_3+e_4)^k$ and obtain that $\mathsf L(A)=L$ by Lemma \ref{4.2}.3. If $k$ is odd, then we set $A=0^y V^2 (e_1+e_4)^{k+1}(e_2+e_4)^{k+1}(e_3+e_4)^{k-1}(e_1+e_2+e_3+e_4)^{k-1}$ and obtain that $\mathsf L(A)=L$ by Lemma \ref{4.2}.3.

Suppose that $L=y+2\lceil\frac{k}{3}\rceil+2+(\{0\}\cup [2,k+2])\in \mathcal L_6$ with $\big(k\ge 5$ or $k=3\big)$ and $y\in \N_0$. If $k\equiv 0\mod 3$, then we set $A=0^y U^{2k/3} V^2$ and hence $\mathsf L(A)=L$ by Equation \eqref{eq9}. If $k\equiv 2\mod 3$, then we set $A=0^y U^{(2k-4)/3} V^4$ and hence $\mathsf L(A)=L$ by \eqref{eq9}. If $k\equiv 1\mod 3$, then we set $A=0^y U^{(2k-8)/3} V^6$ and obtain that $\mathsf L(A)=L$ by Equation \eqref{eq9}.

Suppose that $L\in \mathcal L_7$. If $L=y+2k+3+\{0,1,3\}+3 \cdot [0,k]$ with $y\in \N_0$ and $k\in \N_0$, then we set $A=0^y U^{2k+1} V e_4^2 (e_1+e_2+e_3+e_4)^2$ and obtain that $\mathsf L(A)=L$ by Equation \eqref{eq11}. If $L=y+2k+4+\{0,1,3\}+3 \cdot [0,k]\cup \{y+5k+8\}$ with $y\in \N_0$ and $k\in \N_0$, then we set $A=0^y U^{2k+2} V e_4^2 (e_1+e_2+e_3+e_4)^2$ and obtain that $\mathsf L(A)=L$ by Equation \eqref{eq11}.

Suppose that $L\in \mathcal L_8$. If $L=y+2k+3+\{0,2,3\}+3 \cdot [0,k]$ with $y\in \N_0$ and $k\in \N_0$, then we set $A=0^y U^{2k+2} V$ and hence $\mathsf L(A)=L$ by Equation \eqref{eq10}. If $L=y+2k+4+\{0,2,3\}+3 \cdot [0,k]\cup \{y+5k+9\}$ with $y\in \N_0$ and $k\in \N_0$, then we set $A=0^y U^{2k+3} V e_4^2 (e_1+e_2+e_3+e_4)^2$ and obtain that $\mathsf L(A)=L$ by Equation \eqref{eq10}.

\medskip
\noindent
{\bf Step 2.} We prove that for every $A\in \mathcal B(G^{\bullet})$, $\mathsf L(A)\in  \mathcal L_2\cup \mathcal L_3\cup \mathcal L_4\cup \mathcal L_5\cup \mathcal L_6 \cup \mathcal L_7 \cup \mathcal L_8$.

Let $A \in \mathcal B (G^{\bullet})$. We may suppose that $\Delta (\mathsf L (A)) \ne \emptyset$. By Lemma \ref{4.5}.1 we have to distinguish four cases.

\smallskip
\noindent CASE 1: \,  $\Delta ( \mathsf L (A) ) = \{3\}$.

By Lemma \ref{4.5}, there is a basis  of $G$, say $(e_1, e_2, e_3, e_4)$,  such that $\supp (A) = \{e_1, \ldots, e_4, e_0\}$. Let $n \in \N_0$ be maximal such that $U^{2n} \t A$. Then there exist a proper subset $I \subset [0,4]$, a tuple $(m_i)_{i \in I} \in \N_0^{(I)}$, and $\epsilon \in \{0,1\}$ such that
\[
A =  U^{\epsilon} U^{2n} \prod_{i \in I} (e_i^2)^{m_i} \,.
\]
Using \cite[Lemma 3.6.1]{Ge-Sc16b}, we infer that
\[
\mathsf L (A)  =  \epsilon + \sum_{i \in I} m_i + \mathsf L (U^{2n}) =  \epsilon + \sum_{i \in I} m_i + (2n + 3 \cdot [0, n])\in \mathcal L_2 \,.
\]

\smallskip
\noindent CASE 2: \,  $\Delta ( \mathsf L (A)) = \{1\}$.

Then $\mathsf L (A)$ is an interval, and it is a direct consequence of Proposition \ref{prop_C24_int} that $\mathsf L (A) \in  \mathcal L_3$.

\smallskip
\noindent CASE 3: \,  $\Delta ( \mathsf L(A)) = \{2\}$.

The following reformulation turns out to be convenient. Clearly,
we have to show that for every  $L \in \mathcal{L}(G)$ with $\Delta (L) = \{2\}$ there exist $y' \in \mathbb{N}_0$ and $k' \in \mathbb{N}$ such that  $L=y' + 2 \cdot [k',2k']$, which is equivalent to  $\rho (L)= \max L / \min L  \le 2$.
Assume to the contrary that there is an $L \in \mathcal L (G)$ with $\Delta (L)=\{2\}$ such that $\max L \ge 2 \min L  + 1$. We choose one such $L \in \mathcal L (G)$ with $ \min L$ being minimal, and we choose a $B \in \mathcal B (G)$ with $\mathsf{L}(B)= L$. Since $\min L$ is minimal,  we obtain that $0 \nmid B$. Consequently, $|B| \ge 2  \max L  \ge 4 \min L +2 $. Since $\mathsf D (G) = 5$, it follows that a factorization of minimal length of $B$ contains at least two (possibly  equal) minimal zero-sum sequences $U_1, U_2$ with  $|U_1|=|U_2|=5$, say $U_1 = e_0 \cdot \ldots \cdot e_4$.

If $U_1= U_2$, then $5 \in \mathsf{L} (U_1U_2)$ and thus $\min L + 3 \in L$, contradicting the fact that $\Delta(L)=\{2\}$.
Thus $U_1 \neq U_2$. We assert that $3 \in \mathsf{L} (U_1U_2)$, and thus obtain again a contradiction to the fact that $\Delta(L)=\{2\}$.

Let $g \in G$ with $g \t U_2$ but $g \nmid U_1$. Then $g$ is the sum of two elements from $U_1$, say $g= e_1+ e_2$. Therefore $g(e_1e_2)^{-1}U_1$ is a minimal zero-sum sequence, whereas the sequence
$(e_1e_2)g^{-1}U_2$ cannot be a minimal zero-sum sequence because it has  length  $6$. Since $g^{-1}U_2$ is zero-sum free, every minimal zero-sum sequence dividing   $(e_1e_2)g^{-1}U_2$ must contain $e_1$ or $e_2$. This shows that $\mathsf{L}((e_1e_2)g^{-1}U_2)= \{2\}$ and thus $3 \in \mathsf{L} (U_1U_2)$.

\smallskip
\noindent CASE 4: \,  $\Delta ( \mathsf L (A)) = \{1,2\}$.

Let $k \in \mathsf L (A)$ be minimal such that $A$ has a factorization of the form
$A=U_1\cdot\ldots\cdot U_k=V_1\cdot\ldots\cdot V_{k+2}$, where $k+1\not\in \mathsf L(A)$ and $U_1,\ldots, U_k, V_1,\ldots, V_{k+2} \in \mathcal A (G)$  with $|U_1|\ge |U_2|\ge \ldots \ge |U_k|$.
Without restriction we may suppose that the tuple
\begin{equation} \label{maximal-lexico}
(|\{i\in [1,k]\mid |U_i|=5\}|,|\{i\in [1,k]\mid |U_i|=4\}|,|\{i\in [1,k]\mid |U_i|=3\}|) \in \N_0^3
\end{equation}
is maximal (with respect to the lexicographic order) among all factorizations of $A$ of length $k$.
By definition of $k$, we have $[\min \mathsf L(A), k]\in \mathsf L(A)$.  Let $k_0\in [2,k]$ such that $|U_i|\ge 3$ for every $i\in [1,{k_0}]$ and $|U_i|=2$ for every $i\in [k_0+1,k]$.
We start with the following assertion.

\medskip
\noindent{\bf  A.}
\begin{enumerate}
\item For each two distinct $i,j\in [1,k_0]$, we have  $3\not\in \mathsf L(U_iU_j)$.

\item For each two distinct $i,j\in [1,k_0]$ with $|U_i|=|U_j|=5$, we have  $U_i=U_j$.

\item For each two distinct $i,j\in [1,k_0]$ with $|U_i|=5$ and $|U_j|=4$, we have $| \gcd (U_i, U_j)|=3$.

\item Let $i,j\in [1,k_0]$ be distinct with $|U_i|=|U_j|=4$, say $U_i=f_1f_2f_3(f_1+f_2+f_3)$ where $(f_1,f_2,f_3,f_4)$ a basis of $G$. Then $U_j=U_i$, or $U_j=(f_1+f_4)(f_2+f_4)(f_3+f_4)(f_1+f_2+f_3+f_4)$, or $U_j=f_4(f_1+f_2+f_4)(f_2+f_3+f_4)(f_1+f_3+f_4)$. Furthermore, if $U_i\neq U_j$, then for all $t\in [1, k_0]\setminus \{i,j\}$, we have $|U_t|\neq 4$.

\item Let $i,j\in [1,k_0]$ be distinct with $|U_i|=5$ and $|U_j|=3$. Then there exist $g_1,g_2,g_3\in G$ such that $g_1g_2g_3\t U_i$ and $U_j=(g_1+g_2)(g_2+g_3)(g_3+g_1)$. Furthermore, for all $t\in [1, k_0]\setminus \{i,j\}$, we have $|U_t|=3$.

\item Let $i,j\in [1,k_0]$ be distinct with $|U_i|=4$ and $|U_j|=3$. Then $|\gcd (U_i, U_j)|=0$, and there exist $g,g_1,g_2\in G$ such that $g\t U_j$, $g_1g_2\t U_i$ and $g=g_1+g_2$. Furthermore, for all $t\in [1, k_0]\setminus \{i,j\}$, we have $|U_t|=3$.

\item For each two distinct $i,j\in [1,k_0]$ with $|U_i|=|U_j|=3$, we have $|\gcd (U_i, U_j)|=0$.
\end{enumerate}

\noindent{\it Proof of {\bf  A. }}

1.  If there exist distinct $i,j\in [1,k_0]$ such that $3\in \mathsf L(U_iU_j)$, then $k+1\in \mathsf L(A)$, a contradiction.

\smallskip
2. Since $|U_i|=5$ and $U_j\neq U_i$,  there exist $g,g_1,g_2\in G$ with  $g\t U_j$ and $g_1g_2\t U_i$ such that $g=g_1+g_2$. Thus $U_i(g_1g_2)^{-1}g$ is an atom and $U_jg^{-1}g_1g_1$ is a product of two atoms which implies that $3\in \mathsf L(U_iU_j)$, a contradiction.

\smallskip
3. Since $|U_i|=5$ and $U_j\neq U_i$, there exist  $g,g_1,g_2\in G$ with  $g\t U_j$ and $g_1g_2\t U_i$ such that $g=g_1+g_2$. Thus $gg_1g_2$ is an atom and $U_iU_j(gg_1g_2)^{-1}$ is a sequence of length $6$. By 1., $2\notin \mathsf L(U_iU_j(gg_1g_2)^{-1})$ which implies that $\mathsf L(U_iU_j(gg_1g_2)^{-1})=\{3\}$ and hence $|\gcd (U_i, U_j)|=3$.

\smallskip
4. We set $G_1=\langle f_1,f_2,f_3\rangle$ and distinguish three cases.

Case (i):  $U_j\in \mathcal B(G_1)$.
Since $3\notin \mathsf L(U_iU_j)$, we obtain that $U_j=U_i$.

Case (ii):  $U_j=(g_1+f_4)(g_2+f_4)g_3g_4$ with $g_1g_2g_3g_4\in \mathcal B(G_1)$.

If $g_3,g_4\in \{f_1,f_2,f_3,f_1+f_2+f_3\}$, then $3\in \mathsf L(U_iU_j)$, a contradiction. Thus, without loss of generality, we may assume that $g_3=f_1+f_2\notin \{f_1,f_2,f_3,f_1+f_2+f_3\}$. Thus $g_3f_3(f_1+f_2+f_3)$ is an atom and $(g_1+f_4)(g_2+f_4)f_1f_2g_4$ is a zero-sum sequence of length $5$. Since $3\notin \mathsf L(U_iU_j)$, we have that $(g_1+f_4)(g_2+f_4)f_1f_2g_4$ is an atom of length $5$, a contradiction to the maximality condition in Equation \eqref{maximal-lexico}.

Case (iii):  $U_j=(g_1+f_4)(g_2+f_4)(g_3+f_4)(g_4+f_4)$ with $g_1g_2g_3g_4\in \mathcal B(G_1)$.

First,  suppose that $g_1g_2g_3g_4$ is an atom.  If $g_1g_2g_3g_4\neq U_i$, then there exist an element $h\in \{f_1,f_2,f_3,f_1+f_2+f_3\}$ and distinct $t_1,t_2\in [1,4]$, say $t_1=1, t_2=2$, such that $h=g_1+g_2=(g_1+f_4)+(g_2+f_4)$. Thus $U_ih^{-1}(g_1+f_4)(g_2+f_4)$ is a zero-sum sequence of length $5$ and $h(g_3+f_4)(g_4+f_4)$ is an atom. It follows that $U_ih^{-1}(g_1+f_4)(g_2f_4)$ is atom of length $5$ since $3\notin\mathsf L(U_iU_j)$, a contradiction to the maximality condition in Equation \eqref{maximal-lexico}.
Therefore  $g_1g_2g_3g_4=U_i$ which implies that  $U_j=(f_1+f_4)(f_2+f_4)(f_3+f_4)(f_1+f_2+f_3+f_4)$.

Second, suppose that $g_1g_2g_3g_4$ is not an atom. Without loss of generality, we may assume that $g_1=0$ and $g_2g_3g_4$ is an atom. If $\{g_2,g_3,g_4\}\cap \{f_1,f_2,f_3,f_1+f_2+f_3\}\neq \emptyset$, say $g_2\in \{f_1,f_2,f_3,f_1+f_2+f_3\}$, then $g_2(g_3+f_4)(g_4+f_4)$ is an atom and $U_ig_2^{-1}f_4(g_2+f_4)$ is a zero-sum sequence of length $5$. It follows that $U_ig_2^{-1}f_4(g_2+f_4)$ is atom of length $5$ because $3\notin\mathsf L(U_iU_j)$, a contradiction to the maximality condition in Equation \eqref{maximal-lexico}. Therefore
$\{g_2,g_3,g_4\}\cap \{f_1,f_2,f_3,f_1+f_2+f_3\}=\emptyset$ which implies that $g_2g_3g_4=(f_1+f_2)(f_2+f_3)(f_1+f_3)$ and hence $U_j=f_4(f_1+f_2+f_4)(f_2+f_3+f_4)(f_1+f_3+f_4)$.

\medskip
Now suppose that $U_i\neq U_j$,  and assume to the contrary there exists a $t\in [1,k_0]\setminus\{i,j\}$ such that $|U_t|=4$. If $U_t\notin \{U_i,U_j\}$, then $U_iU_jU_t=\big(f_1f_2f_3(f_1+f_2+f_3)\big)\big( (f_1+f_4)(f_2+f_4)(f_3+f_4)(f_1+f_2+f_3+f_4)\big)\big( f_4(f_1+f_2+f_4)(f_2+f_3+f_4)(f_1+f_3+f_4)\big)=\big(f_1(f_2+f_4)(f_1+f_2+f_4)\big)\big( f_2(f_3+f_4)(f_2+f_3+f_4)\big)\big( f_3(f_1+f_4)(f_1+f_3+f_4) \big)\big( f_4(f_1+f_2+f_3)(f_1+f_2+f_3+f_4)\big)$.  Thus $4\in \mathsf L(U_iU_jU_t)$ and hence $k+1\in \mathsf L(A)$, a contradiction. If $U_t\in \{U_i,U_j\}$, then we still have that $4\in \mathsf L(U_iU_jU_t)$ and hence $k+1\in \mathsf L(A)$, a contradiction.

\smallskip
5. Since $3\notin \mathsf L(U_iU_j)$, we obtain that $|\gcd (U_i, U_j)|=0$. Every $h\in \supp (U_j)$ is the sum of two distinct  elements from $\supp(U_i)$. Thus there exist $g_1,g_2,g_3\in G$ with $g_1g_2g_3\t U_i$ such that $U_j=(g_1+g_2)(g_2+g_3)(g_3+g_1)$.
Now we choose an element $t\in [1,k_0]\setminus\{i,j\}$, and have to show that $|U_t|=3$. If $|U_t|=5$, then $U_t=U_i$ by 2. and hence $4\in \mathsf L(U_iU_tU_j)$ which implies that $k+1\in \mathsf L(A)$, a contradiction. Suppose that $|U_t|=4$ and let $U_i=g_1g_2g_3g_4g_5$, where $g_4,g_5\in G$.  Then $|\gcd(U_i,U_t)|=3$ by 3. and by symmetry we only need to consider $\supp(U_t)\setminus \supp(U_i)\subset \{g_1+g_2, g_1+g_4,g_4+g_5\}$. All the three cases imply that  $4\in \mathsf L(U_iU_tU_j)$. It follows that that $k+1\in \mathsf L(A)$, a contradiction.

\smallskip
6. If  $| \gcd (U_i, U_j)|=2$, then $3\in \mathsf L(U_iU_j)$, a contradiction. If $|\gcd (U_i, U_j)|=1$, then $U_1U_2=W_1W_2$ with $W_1,W_2\in \mathcal A(G)$ and $|W_2|=5$, a contradiction to the maximality condition in Equation \eqref{maximal-lexico}. Thus we obtain that $|\gcd (U_i, U_j)|=0$. Let $(f_1,f_2,f_3,f_4)$ be a basis and $U_i=f_1f_2f_3(f_1+f_2+f_3)$. Since $|U_j|=3$,  there exists a $g \in \supp ( U_j)$ such that $g\in \langle f_1,f_2,f_3\rangle$. Since $|\gcd (U_i, U_j)|=0$, there exist $g_1,g_2\in G$ such that $g_1g_2\t U_i$ and $g=g_1+g_2$.

Now we choose an element $t\in [1,k_0]\setminus\{i,j\}$ and have to show that $|U_t|=3$.  Note that  5. implies that  $|U_t|\neq 5$, and we assume to the contrary that   $|U_t|=4$. Without restriction we may assume that $g=f_1+f_2$, and  by 4., we  distinguish three cases. If $U_t=U_i$, then $f_1^2,f_2^2, gU_i(f_1f_2)^{-1}, U_t(f_1f_2)^{-1}U_jg^{-1}$ are atoms and hence $4\in \mathsf L(U_iU_tU_j)$ which implies that $k+1\in \mathsf L(A)$, a contradiction. If $U_t=(f_1+f_4)(f_2+f_4)(f_3+f_4)(f_1+f_2+f_3+f_4)$, then $g(f_1+f_2+f_3)(f_1+f_2+f_3+f_4)(f_1+f_4)f_2$ is an atom of length $5$ dividing $U_iU_jU_t$ and $U_iU_jU_t(g(f_1+f_2+f_3)(f_1+f_2+f_3+f_4)(f_1+f_4)f_2)^{-1}$ is a product of two atoms, a contradiction to the maximality condition in Equation \eqref{maximal-lexico}. If  $U_t=f_4(f_1+f_2+f_4)(f_2+f_3+f_4)(f_1+f_3+f_4)$, then  $gf_2f_3f_4(f_1+f_3+f_4)$ is an atom of length $5$ dividing $U_iU_jU_t$ and $U_iU_jU_t(gf_2f_3f_4(f_1+f_3+f_4))^{-1}$ is a product of two atoms, a contradiction to the maximality condition in Equation \eqref{maximal-lexico}.

\smallskip
7. If $|\gcd (U_i, U_j)|\ge 2$, then $U_i=U_j$ and hence $3\in \mathsf L(U_iU_j)$ which implies that $k+1\in \mathsf L(A)$, a contradiction.
If $|\gcd (U_i, U_j)|=1$, then $U_iU_j=W_1W_2$ with $W_1,W_2 \in \mathcal A (G)$, $|W_1|=2$, and $|W_2|=4$, a contradiction to the maximality condition in Equation \eqref{maximal-lexico}. Therefore $|\gcd (U_i, U_j)|=0$. This completes the proof of  {\bf A}.
\qed

\medskip
Note that  {\bf  A}.5 implies that   $\{|U_i|\mid i\in [1,k_0]\}\neq \{3,4,5\}$. Thus it remains to discuss the following six subcases.

\medskip
\noindent
CASE 4.1. \,   $\{|U_i|\mid i\in [1,k_0]\}=\{3, 5\}$.

By {\bf  A.}5 and {\bf A.}7, we obtain that $|U_1|=5$, $|U_2|=\ldots =|U_{k_0}|=3$, and that $U_1\cdot\ldots \cdot U_{k_0}$ is square-free. This implies that $\max \mathsf L(U_1\cdot \ldots \cdot U_{k_0})=k_0$, and hence $\max \mathsf L(A)=\max \mathsf L(U_0\cdot \ldots \cdot U_{k_0})+k-k_0=k$, a contradiction.

\smallskip
\noindent
CASE 4.2. \,  $\{|U_i|\mid i\in [1,k_0]\}=\{3,4\}$.

By {\bf  A.}6 and {\bf A.}7, we obtain that $|U_1|=4$, $|U_2|=\ldots =|U_{k_0}|=3$, and that $U_1\cdot\ldots \cdot U_{k_0}$ is square-free. This implies that $\max \mathsf L(U_1\cdot \ldots \cdot U_{k_0})=k_0$, and hence
$\max \mathsf L(A)=\max \mathsf L(U_0\cdot \ldots \cdot U_{k_0})+k-k_0=k$, a contradiction.

\smallskip
\noindent
CASE 4.3. \,  $\{|U_i|\mid i\in [1,k_0]\}=\{3\}$.

By {\bf  A.}7, we obtain that $U_1\cdot\ldots \cdot U_{k_0}$ is square-free. This implies that $\max \mathsf L(U_1\cdot \ldots \cdot U_{k_0})=k_0$, and hence $\max \mathsf L(A)=\max \mathsf L(U_0\cdot \ldots \cdot U_{k_0})+k-k_0=k$, a contradiction.

\smallskip
\noindent
CASE 4.4.   $\{|U_i|\mid i\in [1,k_0]\}=\{5\}$.

By {\bf  A.}2, it follows that $A=U_1^{k_0} U_{k_0+1}\cdot \ldots \cdot U_k$. If $\supp (U_{k_0+1}\cdot \ldots \cdot U_k)\subset \supp(U_1)$, then $\Delta (\mathsf L (A))=\{3\}$, a contradiction. Thus there exists $j\in [k_0+1,k]$ such that $U_j=g^2$ for some  $g\not\in  \supp(U_1)$. Then there exist $g_1,g_2\in G$ such that $g_1g_2\t U_1$ and $g=g_1+g_2$. It follows that  $U_1^2U_j=g_1^2 g_2^2 (U_1(g_1g_2)^{-1}g)^2$, where $g_1^2, g_2^2, U_1(g_1g_2)^{-1}g$ are atoms. Therefore $4\in \mathsf L(U_1^2U_j)$ and hence $k+1\in \mathsf L(A)$, a contradiction.

\smallskip
\noindent
CASE 4.5. \,  $\{|U_i|\mid i\in [1,k_0]\}=\{4\}$.

Assume to the contrary, that $k_0\ge 3$.  Then {\bf  A.}4 implies that $U_1\cdot \ldots \cdot U_{k_0}=U_1^{k_0}$, and we set $G_1=\langle \supp(U_1)\rangle$. If there exists $g\in \supp(U_{k_0+1}\cdot \ldots \cdot U_k)$ such that $g\in G_1\setminus \supp(U_1)$, then $4\in \mathsf L(U_1^2g^2)$ and hence $k+1\in \mathsf L(A)$, a contradiction. If there exist distinct $g_1,g_2\in \supp(U_{k_0+1}\cdot \ldots \cdot U_k)$ such that $g_1\notin G_1$ and $g_2\notin G_1$, then $g_1+g_2\in G_1$. Since $g_1+g_2\in \supp(U_1)$ implies that $5\in \mathsf L(U_1^2g_1^2g_2^2)$ and $k+1\in \mathsf L(A)$, we obtain that $g_1+g_2\in G_1\setminus \supp(U_1)$. Then $U_1^2g_1^2g_2^2=W_1^2W_2W_3$ where $W_1,W_2,W_3\in \mathcal A(G)$ with $|W_1|=4$, $W_1\neq U_1$, and $|W_2|=|W_3|=2$. Say $U_{k_0+1}=g_1^2$ and $U_{k_0+2}=g_2^2$. Then $W_1^2U_3\cdot \ldots \cdot U_{k_0} W_1W_2U_{k_0+3}\cdot\ldots\cdot U_k$ is a factorization of $A$ of length $k$ satisfying the maximality condition of Equation \eqref{maximal-lexico} and hence applying {\bf  A.}4 to this factorization, we obtain a contradiction. Therefore
 $\supp(U_{k_0+1}\cdot \ldots \cdot U_k)\subset \supp(U_1)\cup \{g\}$ where $g$ is independent from $\supp(U_1)$ and hence $\supp(A)\subset \supp(U_1)\cup \{g\}$ which implies that $\Delta (\mathsf L (A))=\{2\}$, a contradiction.

Therefore it follows  that $k_0=2$. Then $U_1=U_2$ (since otherwise we would have $\max\mathsf L(A)=k$ by $U_1U_2$ is square-free), and we obtain that $\mathsf L(A)=[\min \mathsf L(A), k]\cup\{k+2\}$.
Assume to the contrary that  there exists a $W \in \mathcal A (G)$ such that  $W\t A$ and  $|W|=5$. Then there exist $g,g_1,g_2\in G$ such that $g\t U_1$, $g_1g_2\t W$, and $g=g_1+g_2$, and hence  $|\{g_1,g_2\}\cap \supp(U_1)|\le 1$.
If $\{g_1,g_2\}\cap \supp(U_1)=\emptyset$, then there exist distinct $t_1,t_2\in [k_0+1,k]$ such that $U_{t_1}=g_1^2$ and $U_{t_2}=g_2^2$. Thus $5\in \mathsf L(U_1U_2U_{t_1}U_{t_2})$ and hence $k+1\in \mathsf L(A)$, a contradiction.
Suppose that $|\{g_1,g_2\}\cap \supp(U_1)|= 1$, say $g_1\notin \supp(U_1)$ and $g_2\in \supp(U_1)$. Then there exists $t\in [k_0+1,k]$ such that $U_t=g_1^2$. Therefore $4\in \mathsf L(U_1U_2U_t)$ and hence $k+1\in \mathsf L(A)$, a contradiction.

Thus  every atom $W$ with  $W\t A$ has length   $|W| < 5$. It follows that $\min \mathsf L(A)\ge \lceil \frac{2\max\mathsf L(A)}{4}\rceil=\lceil \frac{\max\mathsf L(A)}{2}\rceil $ and hence $\mathsf L(A)\in \mathcal L_5$.

\smallskip
\noindent
CASE 4.6. \,  $\{|U_i|\mid i\in [1,k_0]\}=\{4,5\}$.

By {\bf  A}.2, {\bf  A}.3, and {\bf A.}4, we obtain that $|\{U_1,\ldots, U_{k_0}\}|=2$.  Without restriction we may assume that  $U_1\cdot\ldots\cdot U_{k_0}=U^{k_1}V^{k_2} $ where  $k_1,k_2\in \N$ with $k_0=k_1+k_2$ and $V=e_1e_2e_3(e_1+e_2+e_3)$ (recall that $(e_1, \ldots, e_4)$ is a basis of $G$, $e_0=e_1+e_2+e_3+e_4$, and $U=e_1e_2e_3e_4 e_0$).
We claim that
\begin{itemize}
\item $\supp(U_{k_0+1}\cdot \ldots\cdot U_{k})\subset \supp(UV)$.
\item If $k_1\ge 2$, then $\supp(U_{k_0+1}\cdot \ldots\cdot U_{k})\subset \supp(U)$, and
\item if $k_2\ge 2$, then $\{e_4,e_0\}\not \subset \supp(U_{k_0+1}\cdot \ldots\cdot U_{k})$.
\end{itemize}

Indeed,  assume to the contrary that $g\in \supp(U_{k_0+1}\cdot \ldots\cdot U_{k})\setminus \supp(UV)$. By symmetry, we only need to consider $g=e_1+e_2$ and $g=e_1+e_4$ and both cases imply that $4\in \mathsf L(UVg^2)$, a contradiction to $k+1\notin\mathsf L(A)$. If $k_1\ge 2$ and $g=e_1+e_2+e_3\in \supp(U_{k_0+1}\cdot \ldots\cdot U_{k})$, then $4\in \mathsf L(U^2g^2)$ and $k+1\in \mathsf L(A)$, a contradiction. Thus if $k_1\ge 2$, then $\supp(U_{k_0+1}\cdot \ldots\cdot U_{k})\subset \supp(U)$. If $k_2\ge 2$ and $\{e_4,e_0\} \subset \supp(U_{k_0+1}\cdot \ldots\cdot U_{k})$, then $5\in \mathsf L(V^2e_4^2e_0^2)$ and hence $k+1\in \mathsf L(A)$, a contradiction.

Thus all three claims are proved, and we distinguish three subcases.

\smallskip
\noindent
CASE 4.6.1. \, $k_1=1$.

If $\{e_4,e_0\}\not\subset \supp(U_{k_0+1}\cdot\ldots\cdot U_{k}) $, then $\mathsf L(A)=\mathsf L(UV^{k_2})+k-k_0=\mathsf L(V^{k_0})+k-k_0$ and hence $\Delta (\mathsf L (A))=\{2\}$, a contradiction.
If $\{e_4,e_0\}\subset \supp(U_{k_0+1}\cdot\ldots\cdot U_{k})$, then $k_2=1$ and  we may assume that $U_{k_0+1}=e_4^2$ and that   $U_{k_0+2}=e_0^2$.    Then $\mathsf L(A)=\mathsf L(UVU_{k_0+1}U_{k_0+2})+k-k_0-2=\{k-1,k,k+2\}$ with $k\ge 4$, and hence $\mathsf L(A)\in \mathcal L_5$.

\smallskip
\noindent
CASE 4.6.2. \, $k_1\ge 2$ and $k_2\ge 2$.

Thus $\supp(U_{k_0+1}\cdot\ldots\cdot U_{k})$ is independent  and it follows that  $\supp(U_{k_0+1}\cdot\ldots\cdot U_{k})\subset \{e_1,e_2,e_3,e_4\}$ or $\supp(U_{k_0+1}\cdot\ldots\cdot U_{k})\subset \{e_1,e_2,e_3,e_0\}$.  Then we have $\mathsf L(A)=\mathsf L(U^{k_1}V^{k_2})+k-k_0$.  By  Equation \eqref{eq9}, $\mathsf L(U^{k_1}V^{k_2})$ is equal to
\begin{align*}
\left\{\begin{aligned}
&\{k_0\}\cup [k_0+2, 5\lfloor k_1/2\rfloor+4(k_0/2-\lfloor k_1/2\rfloor)] &&\text{ if $k_0=k_1+k_2$ is even}\,,\\
&\{k_0\}\cup [k_0+2, 5\lfloor k_1/2\rfloor+4((k_0-1)/2-\lfloor k_1/2\rfloor)+1] &&\text{ if $k_0=k_1+k_2$ is odd}\,.
\end{aligned}\right.
\end{align*}
 Let $\ell=\max\mathsf L(U^{k_1}V^{k_2})-k_0-2$ and hence \begin{align*}
\ell=\left\{\begin{aligned}
&k_0+\lfloor\frac{k_1}{2}\rfloor-2 &\text{ if $k_0\ge 4$ is even}\,,\\
&k_0+\lfloor\frac{k_1}{2}\rfloor-3 &\text{ if $k_0\ge 5$ is odd}\,.
\end{aligned}\right.
\end{align*} Since $k_1\ge 2$ and $k_2\ge 2$, we obtain that $\ell\ge 3$ and $\ell\neq 4$.  We also have that
\begin{align*}
\ell\le \left\{\begin{aligned}
&k_0+\lfloor\frac{k_0-2}{2}\rfloor-2=\frac{3k_0}{2}-3 &\text{ if $k_0$ is even}\,,\\
&k_0+\lfloor\frac{k_0-2}{2}\rfloor-3=\frac{3k_0-9}{2} &\text{ if $k_0$ is odd}\,.
\end{aligned}\right.
\end{align*}
Therefore \begin{align*}
k_0\ge \left\{\begin{aligned}
&\frac{2\ell}{3}+2 &\text{ if $k_0$ is even}\,,\\
&\frac{2\ell}{3}+3 &\text{ if $k_0$ is odd}\,,
\end{aligned}\right.
\end{align*}
and hence \begin{align*}k_0\ge \left\{\begin{aligned}
&2\lceil\frac{\ell}{3}\rceil+2 &\text{ if $k_0$ is even}\,,\\
&2\lceil\frac{\ell}{3}\rceil+2 &\text{ if $k_0$ is odd}\,.
\end{aligned}\right.
\end{align*}
It follows that $\mathsf L(U^{k_1}V^{k_2})\in \mathcal L_6$ which implies that
 $\mathsf L(A)\in \mathcal L_6$.

\smallskip
\noindent
CASE 4.6.3. \, $k_1\ge 2$ and $k_2=1$.

Then  $\supp(U_{k_0+1}\cdot\ldots\cdot U_{k})\subset \{e_1,e_2,e_3,e_4, e_0\}$. If $\{e_4,e_0\}\not\subset \supp(U_{k_0+1}\cdot\ldots\cdot U_{k}) $, then
$\mathsf L(A)=\mathsf L(U^{k_1}V)+k-k_0$ is equal to
\begin{align*}
\left\{\begin{aligned}
&k+\{0,2,3\}+3 \cdot [0, k_1/2-1], &&\text{ if $k_1$ is even}\,,\\
&k+\{0,2,3\}+3 \cdot [0,  (k_1-1)/2-1]\cup \{k+(3k_1-3)/2+2\}, &&\text{ if $k_1$ is odd}\,
\end{aligned}\right.
\end{align*}
by Equation \eqref{eq10}.
Therefore $\mathsf L(A)\in \mathcal L_8$.

If $\{e_4,e_0\}\subset \supp(U_{k_0+1}\cdot\ldots\cdot U_{k})$, then we may assume that $U_{k_0+1}=e_4^2$ and that $U_{k_0+2}=e_0^2$. Thus \begin{align*}
\mathsf L(A)&=\mathsf L(U^{k_1}VU_{k_0+1}U_{k_0+2})+k-k_0-2\\
&=\left\{\begin{aligned}
&k-1+\{0,1,3\}+3 \cdot [0, (k_1+1)/2-1], &&\text{ if $k_1$ is odd}\,,\\
&k-1+\{0,1,3\}+3 \cdot [0,  k_1/2-1]\cup \{k+3k_1/2+1\}, &&\text{ if $k_1$ is even}\,,
\end{aligned}\right.
\end{align*}
by Equation \eqref{eq11}
and hence $\mathsf L(A)\in \mathcal L_7$.
 \end{proof}

\medskip
\section{Sets of lengths of weakly Krull monoids} \label{5}
\medskip

It is well-known that -- under reasonable algebraic finiteness conditions --  the Structure Theorem for Sets of Lengths holds for weakly Krull monoids (as it is true for transfer Krull monoids of finite type, see Proposition \ref{3.2}). In spite of this common feature we will demonstrate that systems of sets of lengths for a variety of classes of weakly Krull monoids are different from the  system of sets of lengths of any transfer Krull monoid (apart from well-described exceptional cases; see Theorems \ref{5.5} to \ref{5.8}). Since  half-factorial monoids are transfer Krull monoids, and since there are half-factorial weakly Krull monoids,  half-factoriality is such a natural exceptional case.

So far there are only a couple of results in this direction. In \cite{Fr13a},  Frisch   showed that $\Int (\Z)$, the ring of integer-valued polynomials over $\Z$, is not a transfer Krull domain (nevertheless, the system of sets of lengths of $\Int (\Z)^{\bullet}$ coincides with $\mathcal L (G)$ for an infinite abelian group $G$). To mention a result by Smertnig, let $\mathcal O$ be the ring of integers of an algebraic number field $K$, $A$ a central simple algebra over $K$, and $R$ a classical maximal $\mathcal O$-order of $A$. Then $R$ is a non-commutative Dedekind domain and in particular an HNP ring (see \cite[Sections 5.2 and 5.3]{Mc-Ro01a}). Furthermore, $R$ is a transfer Krull domain if and only if every stably free left $R$-ideal is free (\cite[Theorems 1.1 and 1.2]{Sm13a}).

\smallskip
We gather basic concepts and properties of weakly Krull monoids and domains (Propositions \ref{5.1} and \ref{5.2}). In  the remainder of this section, all monoids and domains are supposed to be commutative.

Let $H$ be a  monoid (hence commutative, cancellative, and with unit element). We denote by $\mathsf q (H)$ the quotient group of $H$, by $H_{\red}=H/H^{\times}$ the associated reduced monoid of $H$, by $\mathfrak X (H)$ the set of minimal nonempty prime $s$-ideals of $H$,  and by $\mathfrak m = H \setminus H^{\times}$ the maximal $s$-ideal. Let $\mathcal I_v^* (H)$ denote the monoid of $v$-invertible $v$-ideals of $H$ (with $v$-multiplication). Then $\mathcal F_v (H)^{\times} = \mathsf q ( \mathcal I_v^* (H))$ is the quotient group of fractional  $v$-invertible $v$-ideals, and $\mathcal C_v (H) = \mathcal F_v (H)^{\times}/\{ xH \mid x \in \mathsf q (H)\}$ is the $v$-class group of $H$ (detailed presentations of ideal theory in commutative monoids can be found in \cite{HK98, Ge-HK06a}). We denote by $\widehat H \subset \mathsf q (H)$ the complete integral closure of $H$, and by $(H \DP \widehat H) = \{ x \in \mathsf q (H) \mid x \widehat H \subset H \} \subset H$ the conductor of $H$. A submonoid $S \subset H$ is said to be saturated if $S = \mathsf q (S) \cap H$. For the definition and discussion of the concepts of being faithfully saturated or being locally tame we refer to \cite[Sections 1.6 and 3.6]{Ge-HK06a}.

To start with the local case, we recall that  $H$ is said to be
\begin{itemize}
\item {\it primary} if $\mathfrak m \ne \emptyset$ and for all $a, b \in \mathfrak m$ there is an $n \in \N$ such that $b^n \subset aH$.

\item  {\it strongly primary} if $\mathfrak m \ne \emptyset$ and for every $a \in \mathfrak m$ there is an $n \in \N$ such that $\mathfrak m^{n} \subset aH$. We denote by $\mathcal M (a)$ the smallest $n$ having this property.

\item a {\it discrete valuation monoid} if it is primary and contains a prime element (equivalently, $H_{\red} \cong (\N_0,+)$).
\end{itemize}
Furthermore,  $H$ is said to be
\begin{itemize}
\item  {\it weakly Krull} (\cite[Corollary 22.5]{HK98}) if
\[
H = \bigcap_{{\mathfrak p} \in \mathfrak X (H)} H_{\mathfrak p}  \quad \text{and} \quad \{{\mathfrak p} \in \mathfrak X (H) \mid a \in {\mathfrak p}\} \quad \text{is finite for all} \ a \in H \,.
\]

\item {\it weakly factorial} if  one of the following equivalent conditions is satisfied (\cite[Exercise 22.5]{HK98}){\rm \,:}
     \begin{itemize}
     \item Every non-unit is a finite product of primary elements.

     \item $H$ is a weakly Krull monoid with trivial $t$-class group.
     \end{itemize}
\end{itemize}
Clearly, every localization $H_{\mathfrak p}$ of  $H$ at a minimal prime ideal $\mathfrak p \in \mathfrak X (H)$ is primary, and a weakly Krull monoid $H$ is $v$-noetherian if and only if  $H_{\mathfrak p}$ is $v$-noetherian for each $\mathfrak p \in \mathfrak X (H)$. Every $v$-noetherian primary monoid $H$ is strongly primary and $v$-local, and if $(H \DP \widehat H)\ne \emptyset$, then $H$ is locally tame  (\cite[Lemma 3.1 and Corollary 3.6]{Ge-Ha-Le07}).
Every strongly primary monoid is a primary \BF-monoid (\cite[Section 2.7]{Ge-HK06a}). Therefore the coproduct of a family of strongly primary monoids is a \BF-monoid,
and every coproduct of a family of primary monoids is weakly factorial. A $v$-noetherian weakly Krull monoid $H$ is weakly factorial if and only if $\mathcal C_v (H)=0$ if and only if $H_{\red} \cong \mathcal I_v^* (H)$.

By a numerical monoid $H$ we mean an additive submonoid of $(\N_0, +)$ such that $\N_0 \setminus H$ is finite. Clearly, every numerical monoid is $v$-noetherian  primary, and hence it is strongly primary.
Note that a numerical monoid is half-factorial if and only if it is equal to $(\N_0,+)$.

Let $R$ be a domain. Then $R^{\bullet} = R \setminus \{0\}$ is a monoid, and all arithmetic and ideal theoretic concepts introduced for monoids will be used for domains in the obvious way.
The domain $R$ is weakly Krull (resp. weakly factorial) if and only if its multiplicative monoid $R^{\bullet}$  is weakly Krull (resp. weakly factorial).
Weakly Krull domains were introduced by Anderson, Anderson, Mott, and Zafrullah (\cite{An-An-Za92b, An-Mo-Za92}). We recall some most basic facts and refer to an extended list of weakly Krull domains and monoids  in \cite[Examples 5.7]{Ge-Ka-Re15a}.
The monoid $R^{\bullet}$ is primary if and only if $R$ is one-dimensional and local. If $R$ is one-dimensional local Mori, then $R^{\bullet}$ is strongly primary and   locally tame (\cite{Ge-Ro18a}). Furthermore, every one-dimensional semilocal Mori domain with nontrivial conductor is weakly factorial and the same holds true for generalized Cohen-Kaplansky domains. It can be seen from the definition that  one-dimensional noetherian domains are $v$-noetherian weakly Krull domains.

Proposition \ref{5.1} summarizes the main algebraic properties of $v$-noetherian weakly Krull monoids and Proposition \ref{5.2} recalls that their arithmetic can be studied via weak transfer homomorphisms to  weakly Krull monoids of very special form.

\medskip
\begin{proposition} \label{5.1}
Let $H$ be a $v$-noetherian weakly Krull monoid.
\begin{enumerate}
\item The monoid $\mathcal I_v^* (H)$ is isomorphic to the coproduct of $(H_{\mathfrak p})_{\red}$ over all $\mathfrak p \in \mathfrak X (H)$. 
In particular, $\mathcal I_v^* (H)$ is weakly factorial and $v$-noetherian.

\smallskip
\item Suppose that  $\mathfrak f = (H \DP \widehat H) \ne \emptyset$. We set $\mathcal P^* = \{ \mathfrak p \in \mathfrak X (H) \mid \mathfrak p \supset \mathfrak f \}$, and $\mathcal P = \mathfrak X (H) \setminus \mathcal P^*$.
    \begin{enumerate}
    \item Then $\widehat H$ is Krull, $\mathcal P^*$ is finite, and  $H_{\mathfrak p}$ is a discrete valuation monoid for each $\mathfrak p \in \mathcal P$. In particular, $\mathcal I_v^* (H)$ is isomorphic to $\mathcal F (\mathcal P) \times \prod_{\mathfrak p \in \mathcal P^*} (H_{\mathfrak p})_{\red}$.

    \smallskip
    \item If $\mathcal H = \{ aH \mid a \in H \}$ is the monoid of principal ideals of $H$, then $\mathcal H \subset \mathcal I_v^* (H)$ is saturated. Moreover, if $H$ is the multiplicative monoid of a domain, then all monoids $H_{\mathfrak p}$ are locally tame and $\mathcal H \subset \mathcal I_v^* (H)$ is faithfully saturated.
    \end{enumerate}
\end{enumerate}
\end{proposition}

\begin{proof}
1. See \cite[Proposition 5.3]{Ge-Ka-Re15a}.

2. For (a) we refer to \cite[Theorem 2.6.5]{Ge-HK06a} and for (b) we refer to  \cite[Theorems 3.6.4 and 3.7.1]{Ge-HK06a}.
 \end{proof}

\medskip
\begin{proposition} \label{5.2}
Let $D = \mathcal F (\mathcal P) \time \prod_{i=1}^n D_i$ be a monoid, where $\mathcal P \subset D$ is a set of primes, $n \in \N_0$, and $D_1, \ldots, D_n$ are reduced primary monoids. Let $H \subset D$ be a saturated submonoid, $G = \mathsf q (D)/\mathsf q (H)$, and $G_\mathcal P = \{ \, p \, \mathsf q (H)  \mid p \in \mathcal P\} \subset G$ the set of classes containing primes.
\begin{enumerate}
\item There is a saturated submonoid $B \subset F=\mathcal F (G_{\mathcal P}) \time \prod_{i=1}^n D_i$ and a weak transfer homomorphism $\theta \colon H \to B$. Moreover, if $G$ is a torsion group, then there is a monomorphism $\mathsf q (F)/\mathsf q (B) \to G$.

\smallskip
\item If $G$ is a torsion group, then $H$ is weakly Krull.
\end{enumerate}
\end{proposition}

\begin{proof}
1. See \cite[Propositions 3.4.7 and 3.4.8]{Ge-HK06a}.

\smallskip
2. See \cite[Lemma 5.2]{Ge-Ka-Re15a}.
 \end{proof}

\medskip
Note that, under the assumption of \ref{5.1}.2, the embedding $\mathcal H \hookrightarrow \mathcal I_v^* (H)$  fulfills the assumptions imposed on the embedding $H \hookrightarrow D$ in Proposition \ref{5.2}. Thus Proposition \ref {5.2} applies to $v$-noetherian  weakly Krull monoids.
For simplicity and in order to avoid repetitions, we formulate the next results (including Theorem \ref{5.7}) in the abstract setting of Proposition \ref{5.2}. However, $v$-noetherian weakly Krull domains and their monoids of $v$-invertible $v$-ideals are in the center of our interest.

If (in the setting of Proposition \ref{5.2}) $G_{\mathcal P}$ is finite, then $F=\mathcal F (G_P)\time \prod_{i=1}^n D_i$ is a finite product of primary monoids and $B \subset F$ is a saturated submonoid. We formulate the main structural result for sets of lengths in $v$-noetherian weakly Krull monoids  in this abstract setting.

\medskip
\begin{proposition} \label{5.3}
Let $D_1, \ldots, D_n$ be locally tame strongly primary monoids and $H \subset D=D_1 \time \ldots \time D_n$ a faithfully saturated submonoid such that $\mathsf q (D)/\mathsf q (H)$ is finite.
\begin{enumerate}
\item The monoid $H$ satisfies the Structure Theorem for Sets of Lengths.

\smallskip
\item There is a finite abelian group $G$ such that for every $L \in \mathcal L (H)$ there is a $y \in \N$ such that $y+L \in \mathcal L (G)$.

\end{enumerate}
\end{proposition}

\begin{proof}
1. follows from  \cite[Theorem 4.5.4]{Ge-HK06a}, and 2. follows from 1. and from Proposition \ref{3.2}.2.
 \end{proof}

The next  lemma on zero-sum sequences will be crucial in order to distinguish between sets of lengths in weakly Krull monoids and sets of lengths in transfer Krull monoids.

\medskip
\begin{lemma} \label{5.4}
Let $G$ be an abelian group and $G_0 \subset G$ a non-half-factorial subset.
\begin{enumerate}
\item  There exists a  half-factorial subset $G_1\subset G_0$ with $\mathcal B (G_1) \ne \{1\}$.

\smallskip
\item There are  $M \in \N$ and   zero-sum sequences $B_k \in \mathcal B (G_0)$ for every $k \in \N$ such that $2 \le |\mathsf L (B_k)| \le M$ but $\min \mathsf L (B_k) \to \infty$ as $k \to \infty$.
\end{enumerate}
\end{lemma}

\begin{proof}
1. Since $G_0$ is not  half-factorial, there is a $B\in \mathcal B(G_0)$ such that $|\mathsf L(B)|>1$. Thus $\supp(B)$ is finite and not half-factorial, say $\supp(B) = \{g_1, \ldots, g_{\ell}\}$ with $\ell \ge 2$. Without restriction we may suppose that  every proper subset of $\{g_1, \ldots, g_{\ell} \}$ is half-factorial.
Assume to the contrary that  for every  subset $G_1\subsetneq \{g_1, \ldots, g_{\ell}\}$ we have $\mathcal B(G_1)=\{1\}$. Since $\{g_1, \ldots, g_{\ell}\}$ is minimal non-half-factorial, there is an atom $A_1 \in \mathcal A(\{g_1, \ldots, g_{\ell}\})$ such that $\mathsf v_{g_i}(A_1)>0$ for every $i \in [1, \ell]$. Since $\{g_1, \ldots, g_{\ell}\}$ is not half-factorial, there is an atom  $A_2 \in \mathcal A(\{g_1, \ldots, g_{\ell}\})$ distinct from $A_1$, say
\[
A_1=g_1^{k_1}\cdot\ldots\cdot g_{\ell}^{k_{\ell}} \quad \text{and} \quad
A_2=g_1^{t_1}\cdot\ldots\cdot g_{\ell}^{t_{\ell}}
\]
 where $k_i \in \N$ and $ t_i \in \N_0$ for every $i\in [1,\ell]$.
Let $\tau\in [1, \ell]$ such that $\frac{t_{\tau}}{k_{\tau}}= \max \{ \frac{t_j}{k_j} \mid j\in [1,\ell] \}$. Then $k_jt_{\tau}-t_jk_{\tau}\ge 0$ for every $j\in [1,\ell]$ whence
\[
W=A_2^{t_{\tau}} A_1^{-k_{\tau}} \in \mathcal B( \{g_1, \ldots, g_{\ell}\} \setminus\{g_{\tau}\})\,,
\]
which implies that $W=1$.
Therefore $\frac{t_{\tau}}{k_{\tau}}=\frac{t_j}{k_j}$ for every $j\in [1,\ell]$ and hence $A_1\t A_2$ or $A_2\t A_1$, a contradiction.

\smallskip
2. Let  $B\in \mathcal B(G_0)$ with $|\mathsf L(B)|>1$. By 1., there exists a half-factorial subset $G_1\subsetneq G_0$ such that $\mathcal B(G_1)\neq \{1\}$. Let $A\in \mathcal A(G_1)$ and $B_k=A^k B$ for every $k\in \N$.
Obviously there exists $k_0\in \N$ such that $\mathsf L(B_k)=\mathsf L(A^{k-k_0})+\mathsf L(B_{k_0})=k-k_0+\mathsf L(B_{k_0})$ for every $k\ge k_0$. Thus $|\mathsf L(B_k)|\le \max \mathsf L(B_{k_0})-\min \mathsf L(B_{k_0})$ and $\min \mathsf L (B_k) = k-k_0 + \min \mathsf L (B_{k_0})$.
 \end{proof}

\smallskip
Now we consider strongly primary monoids and work out a feature of their systems of sets of lengths which does not occur in the system of sets of lengths of any transfer Krull monoid. To do so we study the set $\{ \rho (L) \mid L \in \mathcal L (H) \}$ of  elasticities of all sets of lengths. This set  was studied first by  Chapman et al. in a series of papers (see \cite{B-C-C-K-W06, Ch-Ho-Mo06, B-C-H-M06, Ba-Ne-Pe17a}). Among others they showed that in an atomic monoid $H$, which has a prime element and an element $a \in H$ with $\rho ( \mathsf L (a) ) = \rho (H)$, every rational number $q$ with $1 \le q \le \rho (H)$ can be realized as the elasticity of some $L \in \mathcal L (H)$ (\cite[Corollary 2.2]{B-C-C-K-W06}).
Primary monoids, which are not discrete valuation monoids, have no prime elements and their set of elasticities is different, as we will see in the next theorem.
Statement 1. of Theorem \ref{5.5}  was  proved for numerical monoids  in \cite[Theorem 2.2]{Ch-Ho-Mo06}.

\medskip
\begin{theorem} \label{5.5}
Let $H$ be a strongly primary monoid that is not half-factorial.
\begin{enumerate}
\item There is a $\beta \in \Q_{>1}$ such that $\rho (L) \ge \beta$ for all $L \in \mathcal L (H)$ with $\rho (L) \ne 1$.

\smallskip
\item $\mathcal L (H) \ne \mathcal L (G_0)$ for any subset $G_0$ of any abelian group. In particular, $H$ is not a transfer Krull monoid.

\smallskip
\item If one of the following two conditions holds, then $H$ is locally tame.
      \begin{itemize}
      \smallskip
      \item $\sup \{ \min \mathsf L(c) \mid c \in H \} < \infty$.
      \smallskip
      \item There exists some $u \in H \setminus H^\times$ such that $\rho_{\mathcal M(u)} (H) < \infty$.
      \end{itemize}
      \smallskip
       If $H$ is locally tame, then  $\Delta (H)$ is finite, and there is an $M \in \N_0$ such that every $L \in \mathcal L (H)$ is an {\rm AAMP} with period $\{0, \min \Delta (H)\}$ and bound $M$.
\end{enumerate}
\end{theorem}

\noindent
{\it Remark.} If $H$ is the multiplicative monoid of a one-dimensional local Mori domain $R$ with nonzero conductor $(R \DP \widehat R) \ne \{0\}$, then one of the conditions in 3. is satisfied (see \cite[Proposition 2.10.7 and Theorem 3.1.5]{Ge-HK06a}). However, there are  strongly primary monoids for which none of the conditions holds and which are not locally tame (\cite[Proposition 3.7]{Ge-Ha-Le07}).

\begin{proof}
1.  Let $b\in H$ such that $|\mathsf L(b)|\ge 2$ and let $u\in \mathcal A(H)$. Since $H$ is a strongly primary monoid, we have $(H\setminus H^{\times})^{\mathcal M(b)}\in bH$ and $(H\setminus H^{\times})^{\mathcal M(u)}\in uH$. Thus $b\t u^{\mathcal M(b)}$ and hence $|\mathsf L(u^{\mathcal M(b)})|\ge 2$.  We define
\[
\beta_1=\frac{\mathcal M(b)+\mathcal M(u)+1}{\mathcal M(b)+\mathcal M(u)}\,, \quad  \quad \beta_2=\frac{\max\mathsf L(u^{\mathcal M(b)})+\mathcal M(b)+\mathcal M(u)}{\min\mathsf L(u^{\mathcal M(b)})+\mathcal M(b)+\mathcal M(u)}\,,
\]
and observe that  $\beta = \min \{\beta_1, \beta_2\} >1$. Let $a\in H$ with $\rho (\mathsf L(a)) \ne 1$. We show that $\rho (\mathsf L(a)) \ge \beta$.

Let $k\in \N_0$ be maximal such that $u^k \mid a$, say $a=u^k u'$ with $u'\in H$. Thus $u\nmid u'$ and thus $\max \mathsf L(u')<\mathcal M(u)$. If $k<\mathcal M(b)$, then $\min \mathsf L(a)\le \min \mathsf L(u^k)+\min \mathsf L(u')\le \mathcal M(b)+\mathcal M(u)$, and hence
\begin{align*}
\rho(\mathsf L(a))&=\frac{\max\mathsf L(a)}{\min \mathsf L(a)}\ge \frac{\min\mathsf L(a)+1}{\min \mathsf L(a)}\ge \frac{\mathcal M(b)+\mathcal M(u)+1}{\mathcal M(b)+\mathcal M(u)}=\beta_1\ge \beta\,.
\end{align*}
If $k\ge \mathcal M(b)$, then there exist $t\in \N$ and $t_0\in [0, \mathcal M(b)-1]$ such that $k=t\mathcal M(b)+t_0$, and hence
\begin{align*}
\rho(\mathsf L(a))&=\frac{\max\mathsf L(a)}{\min \mathsf L(a)}\ge \frac{\max\mathsf L(u^k)+\max\mathsf L(u')}{\min \mathsf L(u^k)+\min \mathsf L(u')}  \\
& \ge \frac{t\max\mathsf L(u^{\mathcal M(b)})+\max\mathsf L(u^{t_0})+\max\mathsf L(u')}{t\min \mathsf L(u^{\mathcal M(b)})+\min\mathsf L(u^{t_0})+\min \mathsf L(u')}\\
&\ge \frac{t\max\mathsf L(u^{\mathcal M(b)})+t_0+\max\mathsf L(u')}{t\min \mathsf L(u^{\mathcal M(b)})+t_0+\max \mathsf L(u')} \\
& \ge \frac{t\max\mathsf L(u^{\mathcal M(b)})+\mathcal M (b) + \mathcal M  (u)}{t\min \mathsf L(u^{\mathcal M(b)})+\mathcal M (b) + \mathcal M  (u)} \ge \beta_2\ge \beta\,.
\end{align*}

\smallskip
2. Assume to the contrary that there are an abelian group $G$ and a subset $G_0 \subset G$ such that $\mathcal L (H) = \mathcal L (G_0)$. Since $H$ is not half-factorial, $G_0$ is not half-factorial. By 1., there exists $\beta\in \Q$ with $\beta>1$ such that $\rho(L)\ge \beta$ for every $L \in \mathcal L (H)$.  Lemma \ref{5.4}.2 implies that there are zero-sum sequences $B_k \in \mathcal B (G_0)$ such that $\rho ( \mathsf L (B_k) ) \to 1$ as $k \to \infty$, a contradiction.

\smallskip
3. This follows from  \cite[3.1.1, 3.1.2, and 4.3.6]{Ge-HK06a}.
 \end{proof}

\medskip
Sets of lengths of numerical monoids have found wide attention in the literature (see, among others, \cite{B-C-K-R06, A-C-H-P07a, Co-Ka16a}). As can be seen from Theorem \ref{5.5}.3, the structure of their sets of lengths is simpler than the structure of sets of lengths of transfer Krull monoids over finite abelian groups. Thus it is no surprise that there are infinitely many non-isomorphic numerical monoids whose systems of sets of lengths coincide, and that an analog of Conjecture \ref{3.4} for numerical monoids does not hold true (\cite{A-C-H-P07a}). It is open whether for every $d \in \N$ and every $M \in \N_0$ there is a strongly primary monoid $D$ such that every AAMP with period $\{0,d\}$ and bound $M$ can (up to a shift) be realized as a set of lengths in $D$ (this would be the analog to the realization theorem given in Proposition \ref{3.2}.2). However, for every finite set $L \subset \mathbb N_{\ge 2}$ there is a $v$-noetherian primary monoid $D$ and an element $a \in D$ such that $L = \mathsf L (a)$ (\cite[Theorem 4.2]{Ge-Ha-Le07}).

By Theorem  \ref{3.6} and Proposition \ref{3.2}.3, we know that $\{k, k+1\} \in \mathcal L (G)$ for every $k \ge 2$ and every abelian group $G$ with $|G| \ge 3$. Furthermore, Theorem \ref{3.7} is in sharp contrast to Theorem \ref{5.6}.1.

\medskip
\begin{theorem} \label{5.6}
Let $D = D_1 \times \ldots \times D_n$ be the direct product of strongly primary monoids $D_1, \ldots, D_n$, which are not half-factorial.
\begin{enumerate}
\item For every finite nonempty set $L\subset \N$, there is a $y_L \in \N_0$   such that $y+L \notin \mathcal L (D)$ for any $y \ge y_L$.

\smallskip
\item We have $\mathcal L (D) \ne \mathcal L (G_0)$ for any subset $G_0$ of any abelian group, and hence $D$ is not a transfer Krull monoid. If $D_1, \ldots, D_n$ are locally tame, then $D$ satisfies the Structure Theorem for Sets of Lengths.
\end{enumerate}
\end{theorem}

\begin{proof}
For every $i \in [1,n]$ we choose an element $a_i \in D_i$ such that $|\mathsf L (a_i)| > 1$.

\smallskip
1. Let $L\subset \N$ be a finite nonempty set and
 let $y_L = |L|(\mathcal M (a_1)+ \ldots + \mathcal M (a_n))$.  Assume to the contrary that there are $y \ge y_L$ and    an element $b=b_1 \cdot \ldots \cdot b_n \in D$ such that $\mathsf L (b) = y+L$. Then there is an $i \in [1,n]$ such that $\min \mathsf L (b_i) \ge |L| \mathcal M (a_i)$. Then $b_i \in (D_i \setminus D_i^{\times})^{\min \mathsf L (b_i)} \subset (D_i \setminus D_i^{\times})^{|L| \mathcal M (a_i)} \subset a_i^{|L|} D_i$. Thus there is a $c_i \in D_i$ such that $a_i^{|L|}c_i = b_i$. This implies that $|L|\mathsf L (a_i) + \mathsf L (c_i) \subset \mathsf L (b_i)$. Since $|\mathsf L (a_i)| \ge 2$, we infer that $|\mathsf L (b_i)| \ge |L|+1$ and hence $|L|=|y+L|=|\mathsf L (b)| \ge |\mathsf L (b_i)| \ge |L|+1$, a contradiction.

\smallskip
2. By 1. and Lemma \ref{5.4}.2, the first conclusion follows.

If $D_1, \ldots, D_n$ are locally tame, then $D$ satisfies the Structure Theorem  by Proposition \ref{5.3}.1.
 \end{proof}

\medskip
\begin{theorem} \label{5.7}
Let $D = \mathcal F (\mathcal P) \times D_1$ be the direct product of a free abelian monoid with nonempty basis $\mathcal P$ and of a locally tame strongly primary monoid $D_1$,  and let $G$ be an abelian group. Then $D$ satisfies the Structure Theorem for Sets of Lengths, and the following statements are equivalent{\rm \,:}
\begin{enumerate}
\item[(a)] $\mathcal L (D)= \mathcal L (G)$.

\smallskip
\item[(b)] One of the following  cases holds{\rm \,:}
           \begin{description}
           \smallskip
           \item[(b1)] $|G|\le 2$ and $\rho (D)=1$.

           \smallskip
           \item[(b2)] $G$ is isomorphic either to $C_3$ or to $C_2\oplus C_2$, $[2,3] \in \mathcal L (D)$, $\rho (D)=3/2$,  and $\Delta (D)=\{1\}$.

           \smallskip
           \item[(b3)] $G$ is isomorphic to $C_3 \oplus C_3$, $[2,5] \in \mathcal L (D)$, $\rho (D)=5/2$, and $\Delta (D)=\{1\}$.
           \end{description}
\end{enumerate}
\end{theorem}

\noindent
{\it Remark.} Let $H$ be a $v$-noetherian weakly Krull monoid. If the conductor $(H \DP \widehat H) \in v$-$\max (H)$, then by Proposition \ref{5.1}, $\mathcal I_v^* (H)$ is isomorphic to a monoid $D$ as given in Theorem \ref{5.7}.

\begin{proof}
Since $\mathcal P$ is nonempty, $\mathcal L (D) = \{y+L \mid y \in \N_0, L \in \mathcal L (D_1) \}$ whence $\Delta (D)=\Delta (D_1)$ and $\rho (D)=\rho (D_1)$. In particular, $D$ is half-factorial if and only if $D_1$ is half-factorial. Since $D_1$ satisfies the Structure Theorem of Sets of Lengths by Theorem \ref{5.5}.3,  the same is true for $D$.

If $D$ is half-factorial and  $\mathcal L (D)=\mathcal L (G)$, then $\rho (D)=\rho (D_1)=1$ and $G$ is half-factorial whence $|G| \le 2$ by Proposition \ref{3.3}. Conversely, if $|G|\le 2$ and $\rho (D)=1$, then $G$ and $D$ are half-factorial and $\mathcal L (G)=\mathcal L (D)$.

Thus from now on we suppose that $D_1$ is not half-factorial and that (b1) does not hold. Then $\Delta (D) \ne \emptyset$ and we set $\min \Delta (D)=d$.

\medskip
(a) $\Rightarrow$ (b) Theorem \ref{5.5}.3 and Proposition \ref{3.2}.3 imply that $G$ is finite. Since $G$ is not half-factorial,  it follows that $|G| \ge 3$. Theorem \ref{5.5}.3 shows that  $\Delta_1 (D) = \{d\}$, and since $1 \in \Delta_1 (G) = \Delta_1 (D)$, we infer that $d=1$. Corollary 4.3.16 in \cite{Ge-HK06a} and \cite[Theorem 1.1]{Ge-Zh16a} imply that
\[
\max \{\exp (G)-2, \mathsf r (G)-1\} = \max \Delta_1 (G) = \max \Delta_1 (D) = 1 \,.
\]
Therefore $G$ is isomorphic to one of the following groups: $C_2\oplus C_2$, $C_3$, $C_3 \oplus C_3$. We distinguish two cases.

\smallskip
\noindent
CASE 1: \, $G$ is isomorphic to $C_2 \oplus C_2$ or to $C_3$.

By Proposition \ref{3.3}, we have
\[
\mathcal L (D) = \mathcal L (C_2 \oplus C_2) = \mathcal L (C_3) = \{ y + 2k+[0,k] \mid y,k \in \N_0\} \,.
\]
In particular, we have $3/2 = \rho (G)= \rho (D)$ and $\{1\}=\Delta (G)=\Delta (D)$.

\smallskip
\noindent
CASE 2: \, $G$ is isomorphic to $C_3 \oplus C_3$.

By Theorem \ref{4.1}, just using different notation, we have
\[
\begin{aligned}
\mathcal L (D) = \mathcal L (C_3^2) & = \{ [2k, \ell] \mid k \in \mathbb N_0, \ell \in [2k, 5k]\} \\
 &  \quad \cup \ \{ [2k+1, \ell] \mid k \in \N, \ell \in [2k+1, 5k+2] \} \cup \{ \{ 1\}  \} \,.
\end{aligned}
\]
In particular, we have $5/2=\rho (G)=\rho (D)$ and  $\{1\}=\Delta (G)=\Delta (D)$.

\medskip
(b) $\Rightarrow$ (a) First suppose that Case (b2) holds.  We  show that
\[
\mathcal L (D) =  \bigl\{ y + 2k + [0, k] \, \bigm| \, y,\, k \in \N_0 \bigr\} \,.
\]
Then $\mathcal L (D)=\mathcal L (G)$ by Proposition \ref{3.3}. Since $\rho (D)=3/2$ and $\Delta (D)=\{1\}$, it follows that $\mathcal L (D)$ is contained in the above family of sets. Thus we have to verify that for every $y, k \in \N_0$, the set $y+ [2k,3k] \in \mathcal L (D)$. Since $\mathcal P$ is nonempty, $D$ contains a prime element and hence it suffices to show that $[2k,3k] \in \mathcal L (D)$ for all $k \in \N$. Let $a \in D$ with $\mathsf L (a) = \{2,3\}$, and let $k \in \N$. Then $\min \mathsf L (a^k) \le 2k$ and $\max \mathsf L (a^k) \ge 3k$. Since $\rho ( \mathsf L (a^k)) \le \rho (D) =3/2$, it follows that  $\min \mathsf L (a^k)=2k$ and $\max \mathsf L (a^k)=3k$. Since $\Delta (D)=\{1\}$, we finally obtain that $\mathsf L (a^k)=[2k,3k]$.

Now suppose that Case (b3) holds.  We  show that $\mathcal L (D) $ is equal to
\[
 \{ [2k, \ell] \mid k \in \mathbb N_0, \ell \in [2k, 5k]\} \  \cup \ \{ [2k+1, \ell] \mid k \in \N, \ell \in [2k+1, 5k+2] \} \cup \{ \{ 1\}  \}  \,.
\]
Then $\mathcal L (D)=\mathcal L (G)$ by Theorem  \ref{4.1}. Since $\rho (D)=5/2$ and $\Delta (D)=\{1\}$, it follows that $\mathcal L (D)$ is contained in the above family of sets. Now the proof runs along the same lines as the proof in Case (b2).
 \end{proof}

\medskip
We show that the Cases (b2) and (b3) in Theorem \ref{5.7} can actually occur. Recall that numerical monoids are locally tame and strongly primary.
Let $D_1$ be a numerical monoid distinct from $(\N_0,+)$, say $\mathcal A (D_1)=\{n_1, \ldots, n_t\}$ where $t \in \N_{\ge 2}$ and $1 < n_1 < \ldots < n_t$. Then, by \cite[Theorem 2.1]{Ch-Ho-Mo06} and \cite[Proposition 2.9]{B-C-K-R06},
\[
\rho (D_1) = \frac{n_t}{n_1} \quad \text{and} \quad \min \Delta (D_1) = \gcd(n_2-n_1, \ldots, n_t - n_{t-1}) \,.
\]
Suppose that $\rho (D_1) = m/2$ with $m \in \{3,5\}$ and $\Delta (D_1) =\{1\}$. Then there is an $a \in D_1$ with $\mathsf L (a) = [2,m] \in \mathcal L (D_1)$. Clearly, there are non-isomorphic numerical monoids with elasticity $m/2$ and set of distances equal to $\{1\}$.

\medskip
\begin{theorem} \label{5.8}
Let $R$ be a $v$-noetherian weakly Krull domain with  conductor
 $\{0\} \subsetneq \mathfrak f = (R \DP \widehat R) \subsetneq R$,  and let $\pi \colon  \mathfrak X ( \widehat R) \to \mathfrak X (R)$ be the natural map defined by $\pi ( \mathfrak P) = \mathfrak P \cap R$ for all $\mathfrak P \in \mathfrak X (\widehat R)$.

\begin{enumerate}
\item \begin{enumerate}
      \item $\mathcal I_v^* (H)$ is locally tame with finite set of distances, and it satisfies the Structure Theorem for Sets of Lengths.

      \smallskip
      \item If $\pi$ is not bijective, then $\mathcal L \big(  \mathcal I_v^* (H)   \big) \ne \mathcal L (G_0)$ for any finite subset $G_0$ of any abelian group  and for any subset $G_0$ of an infinite cyclic group. In particular, $\mathcal I_v^* (H)$ is not a transfer Krull monoid of finite type.

      \smallskip
      \item If $R$ is seminormal, then the following statements are equivalent{\rm \,:}
            \begin{enumerate}
            \item $\pi$ is bijective.
            \item $\mathcal I_v^* (H)$ is a transfer Krull monoid of finite type.
            \item $\mathcal I_v^* (H)$ is half-factorial.
            \end{enumerate}
      \end{enumerate}

\smallskip
\item Suppose that the class group $\mathcal C_v (R)$ is finite.
      \begin{enumerate}
      \smallskip
      \item The monoid $R^{\bullet}$ of nonzero elements of $R$ is locally tame with finite set of distances, and it satisfies the Structure Theorem for Sets of Lengths.

      \smallskip
      \item  If $\pi$ is not bijective, then $\mathcal L (R^{\bullet}) \ne \mathcal L (G_0)$ for any finite subset $G_0$ of any abelian group  and for any subset $G_0$ of an infinite cyclic group. In particular, $R$ is not a transfer Krull domain of finite type.

      \smallskip
       \item If $\pi$ is bijective, $R$ is seminormal, every class of $\mathcal C_v (R)$ contains a $\mathfrak p \in \mathfrak X (R)$ with $\mathfrak p \not\supset \mathfrak f$, and the natural epimorphism $\delta \colon \mathcal C_v (R) \to \mathcal C_v(\widehat R)$ is an isomorphism, then there is a weak transfer homomorphism $\theta \colon R^{\bullet} \to \mathcal B ( \mathcal C_v(R))$. In particular, $R$ is a transfer Krull domain of finite type.
      \end{enumerate}
\end{enumerate}
\end{theorem}

\begin{proof}
Since $\mathfrak f \ne R$, it follows that $R \ne \widehat R$ and that $R$ is not a Krull domain. We use the structural description of $\mathcal I_v^* (H)$ as given in Proposition \ref{5.1}.

\smallskip
1.(a) and 2.(a) Both monoids, $R^{\bullet}$ and $\mathcal I_v^* (H)$, are locally tame with finite set of distances by  \cite[Theorem 3.7.1]{Ge-HK06a}. Furthermore,   they both satisfy the Structure Theorem for Sets of Lengths by Proposition \ref{5.3} (use Propositions \ref{5.1} and \ref{5.2}).

\smallskip
1.(b) and 2.(b)  Suppose that  $\pi$ is not bijective. Then $\rho \big( \mathcal I_v^* (H) \big) = \rho (R^{\bullet}) = \infty$ by \cite[Theorems 3.1.5 and  3.7.1]{Ge-HK06a}.
Let $G_0$ be a finite subset of an abelian group $G$. Then $\mathcal B (G_0)$ is finitely generated, the Davenport constant $\mathsf D (G_0)$ is finite  whence the set of distances $\Delta (G_0)$ and the elasticity $\rho (G_0)$ are both finite (see \cite[Theorems 3.4.2 and 3.4.11]{Ge-HK06a}). Thus $\mathcal L \big(  \mathcal I_v^* (H)   \big) \ne \mathcal L (G_0)$ and $\mathcal L (R^{\bullet}) \ne \mathcal L (G_0)$. If $G_0$ is a subset of an infinite cyclic group, then the set of distances is finite if and only if the elasticity is finite by \cite[Theorem 4.2]{Ge-Gr-Sc-Sc10}, and hence the assertion follows again.

\smallskip
1.(c) Suppose that $R$ is seminormal. By 1.(b) and since half-factorial monoids are transfer Krull monoids of finite type, it remains to show that $\pi$ is bijective if and only if $\mathcal I_v^* (H)$ is half-factorial. Since $R$ is seminormal, all localizations $R_{\mathfrak p}$ with $\mathfrak p \in \mathfrak X (H)$ are seminormal. Thus $\mathcal I_v^* (H)$ is isomorphic to a monoid of the form $\mathcal F ( \mathcal P) \time D_1 \time \ldots \time D_n$, where $n \in \N$ and $D_1, \ldots, D_n$ are seminormal finitely primary monoids, and this monoid is half-factorial if and only if each monoid $D_1, \ldots, D_n$ is half-factorial. By \cite[Lemma 3.6]{Ge-Ka-Re15a}, $D_i$ is half-factorial if and only if it has rank one for each $i \in [1,n]$, and this is equivalent to $\pi$ being bijective (see \cite[Theorem 3.7.1]{Ge-HK06a}).

\smallskip
2.(c) This follows from \cite[Theorem 5.8]{Ge-Ka-Re15a}.
 \end{proof}

\smallskip
Note that every order $R$  in an algebraic number field  is a $v$-noetherian weakly Krull domain with finite class group $\mathcal C_v (R)$ such that every class   contains a $\mathfrak p \in \mathfrak X (R)$ with $\mathfrak p \not\supset \mathfrak f$.
If $R$ is a $v$-noetherian weakly Krull domain as above, then Theorems \ref{5.5}, \ref{5.6}, and \ref{5.7} provide further instances of when $R$ is not a transfer Krull domain, but a characterization of the general case remains open. We formulate the following problem (see also \cite[Problem 4.7]{Ge16c}).

\begin{problem} \label{5.9}
Let $H$ be a $v$-noetherian weakly Krull monoid with nonempty conductor $(H \DP \widehat H)$ and finite class group $\mathcal C_v (H)$. Characterize when $H$ and when the monoid $\mathcal I_v^* (H)$ are transfer Krull monoids resp. transfer Krull monoids of finite type.
\end{problem}

\providecommand{\bysame}{\leavevmode\hbox to3em{\hrulefill}\thinspace}
\providecommand{\MR}{\relax\ifhmode\unskip\space\fi MR }
\providecommand{\MRhref}[2]{%
  \href{http://www.ams.org/mathscinet-getitem?mr=#1}{#2}
}
\providecommand{\href}[2]{#2}

\end{document}